\documentclass[a4paper, 10pt, DIV10, headinclude=false, footinclude=false]{scrartcl}

\usepackage{amsthm, amsmath, amssymb, amsfonts}
\usepackage[utf8]{inputenc}
\usepackage[english]{babel}
\usepackage{url}
\usepackage{mathtools}

\usepackage{tikz}
\usepackage{tikz-3dplot}
\tdplotsetmaincoords{70}{110}
\usepackage{pgfplots}
\usepackage{pgfplotstable}
\usepackage{booktabs}

\numberwithin{figure}{section}
\numberwithin{table}{section}
\numberwithin{equation}{section}

\newenvironment{abstr}[1]{ \vspace{.05in}\footnotesize
       \parindent .2in
         {\upshape\bfseries #1. }\ignorespaces}{\par\vspace{.1in}}
\newenvironment{Abstract}{\begin{abstr}{Abstract}}{\end{abstr}}
\newenvironment{keywords}{\begin{abstr}{Key words}}{\end{abstr}}
\newenvironment{AMS}{\begin{abstr}{AMS subject classifications}}{\end{abstr}}

\newtheorem{theorem}{Theorem}[section]
\newtheorem{lemma}[theorem]{Lemma}
\newtheorem{corollary}[theorem]{Corollary}
\newtheorem{proposition}[theorem]{Proposition}
\newtheorem{assumption}[theorem]{Assumption}

\theoremstyle{definition}
\newtheorem{definition}[theorem]{Definition}
\newtheorem{remark}[theorem]{Remark}

\DeclareMathOperator{\diam}{diam}
\DeclareMathOperator{\Div}{div}
\DeclareMathOperator{\curl}{curl}

\DeclareMathOperator{\eff}{eff}
\DeclareMathOperator{\Id}{Id}
\DeclareMathOperator{\stab}{stab}

\DeclareMathOperator{\appr}{appr}

\DeclareMathOperator{\imp}{imp}
\renewcommand{\Re}{\operatorname{Re}}
\renewcommand{\Im}{\operatorname{Im}}
\DeclareMathOperator*{\wto}{\rightharpoonup}
\newcommand{\twosc}{\stackrel{2}{\wto}}
\newcommand{\nz}{\mathbb{N}}       % natural numbers
\newcommand{\gz}{\mathbb{Z}}       % entire numbers
\newcommand{\rz}{\mathbb{R}}       % real numbers
\newcommand{\cz}{\mathbb{C}}       % complex numbers
\newcommand{\pz}{\mathbb{P}}
\newcommand\Va{\mathbf{a}}
\newcommand\Vb{\mathbf{b}}
\newcommand\Vc{\mathbf{c}}
\newcommand\Ve{\mathbf{e}}
\newcommand\Vf{\mathbf{f}}
\newcommand\Vg{\mathbf{g}}

\newcommand\Vn{\mathbf{n}}

\newcommand\Vv{\mathbf{v}}
\newcommand\Vu{\mathbf{u}}
\newcommand\Vw{\mathbf{w}}
\newcommand\Vx{\mathbf{x}}

\newcommand\Vz{\mathbf{z}}

\newcommand\VH{\mathbf{H}}

\newcommand\VV{\mathbf{V}}

\newcommand\Vpsi{\boldsymbol{\psi}}

\newcommand\Vphi{\boldsymbol{\phi}}

\newcommand\Vtheta{\boldsymbol{\theta}}

\newcommand\Vxi{\boldsymbol{\xi}}
\newcommand\CA{\mathcal{A}}
\newcommand\CB{\mathcal{B}}

\newcommand\CH{\mathcal{H}}

\newcommand\CS{\mathcal{S}}
\newcommand\CT{\mathcal{T}}

\allowdisplaybreaks[4]

\begin{document}

\title{Heterogeneous Multiscale Method for the Maxwell equations with high contrast%
\thanks{This work was supported by the Deutsche Forschungsgemeinschaft (DFG) in the project ``OH 98/6-1: Wellenausbreitung in periodischen Strukturen und Mechanismen negativer Brechung''}
}
\author{Barbara Verf\"urth\footnotemark[2]}
\date{}
\maketitle

\renewcommand{\thefootnote}{\fnsymbol{footnote}}
\footnotetext[2]{Angewandte Mathematik: Institut f\"ur Analysis und Numerik, Westf\"alische Wilhelms-Uni\-ver\-si\-t\"at M\"unster, Einsteinstr. 62, D-48149 M\"unster}
\renewcommand{\thefootnote}{\arabic{footnote}}

\begin{Abstract}
In this paper, we suggest a new Heterogeneous Multiscale Method (HMM) for the (time-harmonic) Maxwell scattering problem with high contrast. 
The method is constructed for a setting as in Bouchitt{\'e}, Bourel and Felbacq ({\itshape C.R.\ Math.\ Acad.\ Sci.\ Paris} 347(9-10):571--576, 2009), where the high contrast in the parameter leads to unusual effective parameters in the homogenized equation. 
We present a new homogenization result for this special setting, compare it to existing homogenization approaches and analyze the stability of the two-scale solution with respect to the wavenumber and the data. 
This includes a new stability result for solutions to time-harmonic Maxwell's equations with matrix-valued, spatially dependent coefficients.
The HMM is defined as direct discretization of the two-scale limit equation.
 With this approach we are able to show quasi-optimality and a priori error estimates in energy and dual norms under a resolution condition that inherits its dependence on the wavenumber from the stability constant for the analytical problem.
 This is the first wavenumber-explicit resolution condition for time-harmonic Maxwell's equations.
Numerical experiments confirm our theoretical convergence results.
\end{Abstract}



\section{Introduction}
\label{sec:introduction}

The interest in (locally) periodic media, such as photonic crystals, has grown in the last years as they exhibit astonishing properties such as band gaps or negative refraction, see \cite{EP04negphC, PE03lefthanded, CJJP02negrefraction}. 
In this paper, we extend the study of {\itshape artificial magnetism} from the two-dimensional case in \cite{OV16hmmhelmholtz} to the full three-dimensional case.
Artificial magnetism describes the occurrence of an (effective) permeability $\mu\neq 1$ in an originally non-magnetic material, i.e.\ $\mu=1$.
The study of the two-dimensional reduction, the Helmholtz equation, in \cite{BF04homhelmholtz} has shown that such a material must exhibit a high contrast structure (see below) to allow this significant change of behavior.
The homogenization analysis has been extended to the full three-dimensional Maxwell equations in \cite{BBF09hommaxwell, BBF15hommaxwell} to obtain a wavenumber-dependent effective permeability, which can even have a negative real part.
The frequencies where the real part of the permeability is negative are of particular interest as they form the band gap: Wave propagation is forbidden in these cases.
Although producing the same qualitative results, we emphasize that there are significant differences from the two- to the three-dimensional case, for instance that the effective permeability also depends on the solution outside the inclusions (see below).
The setting of \cite{BBF09hommaxwell} can be complemented with long and thin wires as in \cite{LS15negindex} to obtain a negative effective $\varepsilon$ as well and thus a negative refractive index. 

The setting of \cite{BBF15hommaxwell} is the following (see Figure \ref{fig:setting}):
A periodic structure of three-dimensional bulk inclusions with high permittivity (depicted in gray in Figure \ref{fig:setting}) is embedded in a lossless dielectric material.
Denoting by the small parameter $\delta$ the periodicity, the high permittivity in the inclusions is modeled by setting $\varepsilon^{-1}=\delta^2\varepsilon_1^{-1}$, see \eqref{eq:hetparam} for an exact definition.
The consideration of small inclusions with high permittivity has become a popular modeling to tune unusual effective material properties, see \cite{BBF09hommaxwell, BS10splitring, CC15hommaxwell, LS15negindex}.

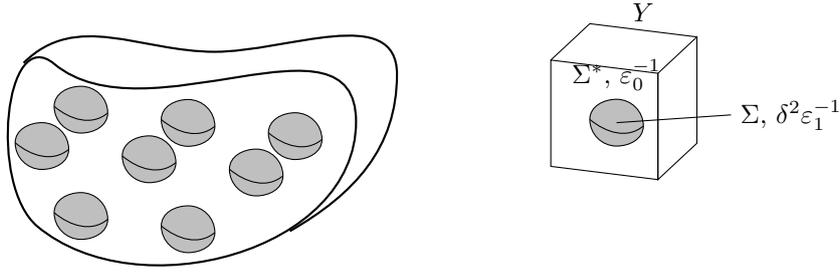
\begin{figure}
\centering
\begin{tikzpicture}[tdplot_main_coords, scale=1.5]
\draw[thick] (0,0,0) to [out=140, in=140] (4,1.3,0)
to [out=-40, in=270] (0.5,3,0)
to [out=90, in=320] (0,0,0);
\draw[thick] (0.07,-0.26,0) to [out=45, in=180] (0, 1.5, 0.3)
to [out=0, in=90] (0, 3.2, 0.3)
to [out=270, in=390] (0, 2.2, -1.3);
% object1
\draw[fill=lightgray] (0,0, -0.4) to [out=90, in=180] (0, 0.25, -0.18)
to [out=0, in=90] (0, 0.5, -0.4)
to [out=270, in=360] (0, 0.25, -0.62)
to [out=180, in=270] (0, 0, -0.4);
\draw (0, 0, -0.4) to [out=-40, in =220] (0, 0.5, -0.4);
%object1 +x_shift
\draw[fill=lightgray] (1.0,0, -0.4) to [out=90, in=180] (1.0, 0.25, -0.18)
to [out=0, in=90] (1.0, 0.5, -0.4)
to [out=270, in=360] (1.0, 0.25, -0.62)
to [out=180, in=270] (1.0, 0, -0.4);
\draw (1.0, 0, -0.4) to [out=-40, in =220] (1.0, 0.5, -0.4);
%object1+y_shift
\draw[fill=lightgray] (0,1.0, -0.4) to [out=90, in=180] (0, 1.25, -0.18)
to [out=0, in=90] (0, 1.5, -0.4)
to [out=270, in=360] (0, 1.25, -0.62)
to [out=180, in=270] (0, 1.0, -0.4);
\draw (0, 1.0, -0.4) to [out=-40, in =220] (0, 1.5, -0.4);
%object1+xy-shift
\draw[fill=lightgray] (1.0,1.0, -0.4) to [out=90, in=180] (1.0, 1.25, -0.18)
to [out=0, in=90] (1.0, 1.5, -0.4)
to [out=270, in=360] (1.0, 1.25, -0.62)
to [out=180, in=270] (1.0, 1.0, -0.4);
\draw (1.0, 1.0, -0.4) to [out=-40, in =220] (1.0, 1.5, -0.4);
% object1-z_shift
\draw[fill=lightgray] (0,0, -1.4) to [out=90, in=180] (0, 0.25, -1.18)
to [out=0, in=90] (0, 0.5, -1.4)
to [out=270, in=360] (0, 0.25, -1.62)
to [out=180, in=270] (0, 0, -1.4);
\draw (0, 0, -1.4) to [out=-40, in =220] (0, 0.5, -1.4);
% object1-z_shift+y_shift
\draw[fill=lightgray] (0,1.0, -1.4) to [out=90, in=180] (0, 1.25, -1.18)
to [out=0, in=90] (0, 1.5, -1.4)
to [out=270, in=360] (0, 1.25, -1.62)
to [out=180, in=270] (0, 1.0, -1.4);
\draw (0, 1.0, -1.4) to [out=-40, in =220] (0, 1.5, -1.4);
%object1+2*y_shift
\draw[fill=lightgray] (0,2.0, -0.4) to [out=90, in=180] (0, 2.25, -0.18)
to [out=0, in=90] (0, 2.5, -0.4)
to [out=270, in=360] (0, 2.25, -0.62)
to [out=180, in=270] (0, 2.0, -0.4);
\draw (0, 2.0, -0.4) to [out=-40, in =220] (0, 2.5, -0.4);
%object1+2*y_shift+x_shift
\draw[fill=lightgray] (1.0,2.0, -0.4) to [out=90, in=180] (1.0, 2.25, -0.18)
to [out=0, in=90] (1.0, 2.5, -0.4)
to [out=270, in=360] (1.0, 2.25, -0.62)
to [out=180, in=270] (1.0, 2.0, -0.4);
\draw (1.0, 2.0, -0.4) to [out=-40, in =220] (1.0, 2.5, -0.4);
%zoom into unit cell
%
\draw (1,5,0) --(1,6,0)-- (1,6,1) --(1,5,1)--cycle;
\draw (1,6,0) --(0,6,0)-- (0,6,1)--(1,6,1)--cycle;
\draw (1,5,1) -- (0,5,1)--(0,6,1);
\node [above] at (0, 5.5, 1) {$Y$};
% object1
\draw[fill=lightgray] (0,5.0, 0.1) to [out=90, in=180] (0, 5.25, 0.32)
to [out=0, in=90] (0, 5.5, 0.1)
to [out=270, in=360] (0, 5.25, -0.12)
to [out=180, in=270] (0, 5.0, 0.1);
\draw (0, 5.0, 0.1) to [out=-40, in =220] (0, 5.5, 0.1);
\node [above] at (0, 5.25, 0.32) {$\Sigma^*$, $\varepsilon^{-1}_0$};
\draw (0, 5.25, 0.1) -- (0.5, 6.5, 0.5);
\node [right] at (0.5, 6.5, 0.5) {$\Sigma$, $\delta^2\varepsilon^{-1}_1$};
\end{tikzpicture}
\caption{Left: Scatterer $\Omega$ with the high contrast inclusions $\Sigma_\delta$ (in gray); Right: Zoom into unit cell $Y$ and scaling of $\varepsilon^{-1}_\delta$}
\label{fig:setting}
\end{figure}

The overall setting in this paper is as follows (cf.\ \cite{BBF09hommaxwell, BBF15hommaxwell}): We consider a scatterer  $\Omega\subset\rz^3$ bounded and smooth (with $C^2$ boundary). The structure is non-magnetic, i.e.\ $\mu=1$, and has a (relative) permittivity $\varepsilon$, which equals $1$ outside $\Omega$.
The magnetic field $\VH$ now solves the following curl-curl-problem
\begin{equation}
\label{eq:scatteringstrong}
\curl \varepsilon^{-1}\curl \VH\ = k^2 \VH,
\end{equation}
where $k=\omega/c$ is the (fixed) wavenumber.
Originally, this problem is studied on the whole space $\rz^3$, complemented with Silver-M\"uller radiation conditions at infinity, see e.g.\ \cite{BBF15hommaxwell}.
Here, we artificially truncate the computational domain, by introducing a large and smooth domain $G\supset\Omega$ and imposing the following impedance boundary condition
\begin{equation}
\label{eq:impedancebdry}
\curl \VH\times \Vn-ik(\Vn\times \VH)\times \Vn=\Vg \text{ on }\partial G
\end{equation}
with a tangential vector field $\Vg$ coming from the incident wave.
The permittivity $\varepsilon^{-1}=\varepsilon_\delta^{-1}$ inside the scatterer models the described setting of periodic inclusions with high permittivity and is defined in \eqref{eq:hetparam}.
Throughout this article, we assume that there is $k_0>0$ such that $k\geq k_0$, which corresponds to medium and high frequencies.

A numerical treatment of \eqref{eq:scatteringstrong} with boundary condition  \eqref{eq:impedancebdry} and permittivity with high contrast is very challenging. 
The main challenge is to well approximate the heterogeneities in the material and the oscillations induced by the incoming wave.
It is important to relate the scales of these oscillations: We basically have a three-scale structure here with $\delta\ll \lambda\sim k^{-1}<1$, i.e.\ the periodicity of the material (and the size of the inclusions) is much smaller than the wavelength of the incoming wave.
A direct discretization requires a grid with mesh size $h<\delta\ll 1$ to approximate the solution faithfully. This can easily exceed today's computational resources when using a standard approach. In order to make a numerical simulation feasible, so called multiscale methods can be applied.
 The family of Heterogeneous Multiscale Methods (HMM) \cite{EE03hmm, EE05hmm} is a class of multiscale methods that has been proved to be very efficient for scale-separated locally periodic problems. The HMM can exploit local periodicity in the coefficients to solve local sample problems that allow to extract effective macroscopic features and to approximate solutions with a complexity independent of the (small) periodicity $\delta$. 
First analytical results concerning the approximation properties of the HMM for elliptic problems have been derived in \cite{Abd05hmmanalysis, EMZ05hmmanalysis, Ohl05HMM} and then extended to other problems, such as time-harmonic Maxwell's equations \cite{HOV15maxwellHMM} and the Helmholtz equation with high contrast \cite{OV16hmmhelmholtz}.
Another related work is the multiscale asymptotic expansion for Maxwell's equations \cite{CZAL10multiscalemaxwell}.

The new contribution of this article is the first formulation of a Heterogeneous Multiscale Method for the Maxwell scattering problen with high contrast in the setting of \cite{BBF15hommaxwell}, its comprehensive numerical analysis and its implementation. 
The HMM can be used to approximate the true solution to \eqref{eq:scatteringstrong} with a much coarser mesh and hence less computational effort.
From the theoretical point of view, the main result is that the energy error converges with rate $k^{q+1}(H+h)+k^{q+1/2}H^{1/2}$ if the resolution condition $k^{q+2}(H+h)+k^{q+3/2}H^{1/2}=O(1)$ is fulfilled. 
Here, $H$ and $h$ denote the $\delta$-independent mesh sizes used for the HMM and we assume that the analytical two-scale solution has a stability constant of order $k^q$ with $q\in \nz_0$.
This is also -- to the author's best knowledge- -- first $k$-explicit resolution condition result for indefinite time-harmonic Maxwell's equations.
The existing literature \cite{GM12maxwellimpedance, Hipt02FEem,Hipt15maxwellcontdiscr, Monk} so far has only shown well-posedness and quasi-optimality for sufficiently fine meshes, without specifying the dependence of this threshold on $k$.
This stands in sharp contrast to the vast literature on the resolution condition for the Helmholtz equation, see e.g.\ \cite{MS11helmholtz, Sau06convanahelmholtz}.
A major issue for the analysis is the large kernel of the curl-operator implying that the $L^2$-identity term is no compact perturbation of the curl-term and that we cannot expect macroscopic functions to be good approximations in $L^2$, see \cite{GHV17lodmaxwell}.

To complement our numerical analysis, we also show an explicit stability estimate for the solution to the two-scale limit equation, so that we have an explicit (though maybe sub-optimal) result for the stability exponent, namely $q=3$.
This includes  a second contribution, which may be of own interest: a new stability result for a certain class of time-harmonic Maxwell's equations, namely with  matrix-valued spatially dependent coefficients. 
Stability results for Maxwell's equations with impedance boundary conditions have so far been only shown in the case of constant coefficients in \cite{FW14dgmaxwell, HMP11stabilityMaxwell, Moiola}.

The paper is organized as follows: In Section \ref{sec:problem}, we detail the (geometric) setting of the problem to be considered and introduce basic notation used throughout this article.
In Section \ref{sec:homogenization}, we give the homogenization results obtained for this problem in form of a two-scale and an effective macroscopic equation. 
These homogenized systems are analyzed with respect to stability and regularity in Section \ref{sec:analysis}.
In Section \ref{sec:numerics}, we introduce the Heterogeneous Multiscale Method and perform a rigorous a priori error analysis.
The main proofs are given in Section \ref{sec:proofs}.
A numerical experiment is presented in Section \ref{sec:experiment}.

\section{Problem setting}
\label{sec:problem}
For the remainder of this article, let $\Omega\subset\subset G\subset \rz^3$ be bounded, simply connected domains with $C^2$ boundary, $G$ with outer unit normal $\Vn$.
Vector-valued functions are indicated by boldface letters and unless otherwise stated, all functions are complex-valued.
Throughout this paper, we use standard notation:
For a domain $D$, $p\in[1,\infty)$ and $s\in \rz_{\geq 0}$, $L^p(D)$ denotes the usual complex Lebesgue space with norm $\|\cdot \|_{L^p(D)}$.
By $W^{s,p}(D)$ we denote  the space of functions on $D$ with (fractional) weak derivatives up to order $s$ belonging to $L^p(D)$ and we write $H^s(D):=W^{s,2}(D)$ for the scalar and $\mathbf{H}^s(D):=[H^s(D)]^3$ for the vector-valued case. 
The domain $D$ is omitted from the norms if no confusion can arise.
The dot will denote a normal (real) scalar product, for a complex scalar product we will explicitly conjugate the second component by using $v^*$ as the conjugate complex of $v$. 
Furthermore, we introduce the Hilbert spaces
\begin{equation*}
\begin{split}
\VH(\curl,D)&:=\{\Vu\in L^2(D; \cz^3)| \hspace{2pt} \curl \Vu\in L^2(D; \cz^3)\} \qquad \mbox{and}\\
\VH(\Div,D)&:=\{\Vu\in L^2(D; \cz^3)| \hspace{2pt} \Div \Vu\in L^2(D; \cz)\}
\end{split}
\end{equation*}
with their standard scalar products $(\cdot, \cdot)_{\VH(\curl, D)}$ and $(\cdot, \cdot)_{\VH(\Div, D)}$, respectively.
In order to define a suitable function space for the scattering problem, we introduce the following space of tangential $L^2$ functions on the boundary
\[L^2_T(\partial G):=\{\Vv\in [L^2(\partial G)]^3|\Vv\cdot \Vn=0\}.\]
We denote by $\Vu_T:=(\Vn\times \Vu)\times \Vn=\Vu-(\Vu\cdot \Vn)\Vn$ the tangential component of a vector function $\Vu$ on the boundary.
Now we define the space for the impedance boundary condition as
\[\VH_{\mathrm{imp}}(G):=\{\Vu\in \VH(\curl, G)|\Vu_T\in L^2_T(\partial G)\}\]
equipped with the graph norm, see \cite{Monk}.
We will frequently replace the standard norms of $\VH(\curl)$ and $\VH_{\mathrm{imp}}$ by the equivalent weighted norms
\begin{align*}
\|\Vv\|_{\curl;k;D}&:=(\|\curl\Vv\|_{L^2(D)}^2+k^2\|\Vv\|^2_{L^2(D)})^{1/2}\\
\text{and}\quad\|\Vv\|_{\mbox{\tiny{imp}};k;D}&:=(\|\curl\Vv\|_{L^2(D)}^2+k^2\|\Vv\|^2_{L^2(D)}+k\|\Vv_T\|^2_{L^2(\partial D)})^{1/2}.
\end{align*}
 To quantify higher regularity, we define for $s\in \nz_0$ the space
\begin{equation*}
\VH^s(\curl, D):=\{\Vu\in \VH(\curl, D)\, |\; \Vu\in \VH^s(D), \curl\Vu\in \VH^s(D)\}.
\end{equation*}
Observe that $\VH^0(\curl)=\VH(\curl)$.

Let $\Ve_j$ denote the $j$'th unit vector in $\rz^3$. 
For the rest of the paper we write $Y:=[-\frac 12, \frac 12)^3$ to denote the 3-dimensional unit cube and we say that a function $v\in L^2_{\mbox{\tiny loc}}(\mathbb{R}^3)$ is $Y$-periodic if it fulfills $v(y)=v(y+\Ve_j)$ for all $j=1,2,3$ and almost every $y\in \mathbb{R}^3$. 
With that we denote $L_{\sharp}^2(Y):=\{ v \in L^2_{\mbox{\tiny loc}}(\mathbb{R}^3)| \hspace{2pt} v \mbox{ is $Y$-periodic}\}$. Analogously we indicate periodic function spaces by the subscript $\sharp$. 
For example, $H^1_\sharp(Y)$ is the space of periodic $H^1_{\mbox{\tiny loc}}(\mathbb{R}^3)$-functions and we furthermore define for $s\in\mathbb{N}$
\[H^s_{\sharp,0}(Y):=\left\{ \left.\phi \in H^s_\sharp(Y)\right|\int_Y \phi(y) \, dy =0\right\}.\]
For $\Sigma^*\subset Y$, we denote by $H^1_{\sharp, 0}(\Sigma^*)$ and $\VH_\sharp(\curl, \Sigma^*)$ the restriction of functions in $H^1_{\sharp, 0}(Y)$ and $\VH_\sharp(\curl, Y)$ to $\Sigma^*$, respectively.
By $L^p(\Omega; X)$ we denote Bochner-Lebesgue spaces over the Banach space $X$ and we use the short notation $f(x,y):=f(x)(y)$ for $f\in L^p(\Omega; X)$.

\smallskip
Using the above notation we consider the following setting for the (inverse) relative permittivity $\varepsilon^{-1}$, see \cite{BBF15hommaxwell}:
$\Omega$ is composed of $\delta$-periodically disposed bulk inclusions, $\delta$ being a small parameter. Denoting by $\Sigma\subset\subset Y$ a connected domain with $C^2$ boundary, the inclusions occupy a region $\Sigma_\delta=\cup_{j\in I} \delta(j+\Sigma)$ with $I=\{j\in \gz^3|\delta(j+Y)\subset \Omega\}$. 
The complement of $\Sigma$ in $Y$, which has to be simply-connected, is denoted by $\Sigma^*$. 
The inverse relative permittivity $\varepsilon^{-1}_\delta=\varepsilon^{-1}$ is then defined (possibly after rescaling) as (cf.\ Figure \ref{fig:setting})
\begin{align}
\label{eq:hetparam}
\varepsilon_\delta^{-1}(x):= 
\begin{cases}
\delta^2\varepsilon_1^{-1}(\frac{x}{\delta}) & \text{if } x\in \Sigma_\delta\qquad\quad\;\; \text{ with }\varepsilon_1^{-1}\in L^\infty_\sharp(Y, \cz);\, \Im(\varepsilon_1)>0, \Re(\varepsilon_1)>0,\\
\varepsilon_0^{-1}(\frac{x}{\delta}) & \text{if }x\in \Omega\setminus \Sigma_\delta\qquad \text{ with }\varepsilon_0\in L^\infty_\sharp(Y, \rz) \text{ uniformly positive},\\
1 & \text{if }x \in G\setminus\overline{\Omega}.
\end{cases}
\end{align}
We assume $\Re(\varepsilon_1)>0$ for simplicity; all results hold -- up to minor modifications in the proofs -- also for $\varepsilon_1$ with $\Re(\varepsilon_1)\leq 0$. 
Physically speaking, this means that the scatterer $\Omega$ consists of periodically disposed metallic inclusions $\Sigma_\delta$ embedded in a dielectric ``matrix'' medium. 

\begin{definition}
Let $\varepsilon_\delta^{-1}$ be defined by \eqref{eq:hetparam} and let $\Vg\in L^2_T(\partial G)$.
The weak formulation of \eqref{eq:scatteringstrong} is: Find $\Vu_\delta\in \VH_{\imp}(G)$ such that
\begin{equation}
\label{eq:scatteringweak}
\int_G \varepsilon_\delta^{-1}\curl\Vu_\delta\cdot \curl \Vpsi^*-k^2\Vu_\delta\cdot \Vpsi^*\, dx-ik\int_{\partial G}(\Vu_\delta)_T\cdot \Vpsi_T^*\, d\sigma=\int_{\partial G}\Vg\cdot \Vpsi_T^*\, d\sigma \qquad \forall \Vpsi\in \VH_{\imp}(G).
\end{equation}
\end{definition}
The problem admits a unique solution for fixed $\delta$, which can be shown with the Fredholm theory, see e.g.\ \cite[Theorem 4.17]{Monk}.
Throughout the article, $C$ denotes a generic constant, which does not depend on $k$, $H$, or $h$, and we use the notation $a\lesssim b$ for $a\leq C b$ with such a generic constant.

\section{Homogenization}
\label{sec:homogenization}
As the parameter $\delta$ is very small in comparison to the wavelength and the typical length scale of $G$, one can reduce the complexity of problem \eqref{eq:scatteringweak} by considering the limit $\delta\to 0$.
This process, called homogenization, can be performed with the tool of two-scale convergence \cite{All92twosc}.
It has also been used in the papers \cite{BBF09hommaxwell, BBF15hommaxwell, CC15hommaxwell} studying closely related problems/formulations.
We proceed in a slightly different way and provide our homogenization results in Subsection \ref{subsec:twoscale}.
In Subsection \ref{subsec:literaturecomp}, we compare with the mentioned literature and show the equivalence of various formulations.

In addition to the notation from Section \ref{sec:problem}, we introduce the space 
\[\widetilde{\VH}_{\sharp}(\curl, \Sigma^*):= \VH_{\sharp}(\curl, \Sigma^*)/\ker (\curl_{y|_{\Sigma^*}}).\]
This is the space of functions $\Vv\in \VH_{\sharp}(\curl, \Sigma^*)$ such that $\curl_y\Vv$ is uniquely determined in $\Sigma^*$ or, in other words, such that $\Vv$ is determined up to a gradient (as $\Sigma^*$ is simply connected).
Note, however, that in practical applications, we will always be interested in $\curl_y \Vv$ only, which is in $L^2_\sharp(\Sigma^*)$ and uniquely determined.

\subsection{Two-scale and effective equations}
\label{subsec:twoscale}
Two-scale convergence is defined and characterized in \cite{All92twosc}, for instance. We write in short form $\twosc$.
The special scaling of $\varepsilon_\delta^{-1}$ leads to a different behavior of the solution inside $\Sigma_\delta$, which can be seen in the two-scale equation and the homogenized effective equation.

\begin{theorem}[Two-scale equation]
\label{thm:twoscaleeq}
Let $\Vu_\delta$ be the unique solution of \eqref{eq:scatteringweak}.
There are functions $\Vu\in \VH_{\imp}(G)$, $\Vu_1\in L^2(\Omega; \widetilde{\VH}_{\sharp}(\curl, \Sigma^*))$, $u_2\in L^2(\Omega; H^1_{\sharp, 0}(\Sigma^*))$, and $\Vu_3\in L^2(\Omega; \VH_0(\curl, \Sigma))$, such that the following two-scale convergences hold 
\begin{align*}
\Vu_\delta &\twosc \Vu+\chi_{\Sigma^*}\nabla_y u_2 +\chi_\Sigma\Vu_3,&&&
\chi_{\Omega\setminus \Sigma_\delta}\curl\Vu_\delta&\twosc \chi_{\Sigma^*}(\curl \Vu +\curl_y \Vu_1),\\
\delta\chi_{\Sigma_\delta} \curl \Vu_\delta&\twosc \chi_\Sigma \curl_y \Vu_3,&&&
\curl \Vu_\delta&\twosc \curl\Vu\quad \text{in}\quad G\setminus\overline{\Omega}.
\end{align*}
The quadruple $\underline{\Vu}:=(\Vu, \Vu_1, u_2, \Vu_3)\in \CH$ of two-scale limits is the unique solution to
\begin{equation}
\label{eq:twoscaleeq}
\CB((\Vu, \Vu_1, u_2, \Vu_3), (\Vpsi, \Vpsi_1, \psi_2, \Vpsi_3))=(\Vg, \Vpsi)_{\partial G} \qquad \forall (\Vpsi, \Vpsi_1, \psi_2, \Vpsi_3)\in \CH
\end{equation}
with $\CH:=\VH_{\imp}(G)\times L^2(\Omega; \widetilde{\VH}_\sharp(\curl, \Sigma^*))\times L^2(\Omega; H^1_{\sharp, 0}(\Sigma^*))\times L^2(\Omega; \VH_0(\curl, \Sigma))$ and
\begin{align*}
&\!\!\!\!\CB((\Vv, \Vv_1, v_2, \Vv_3), (\Vpsi, \Vpsi_1, \psi_2, \Vpsi_3))\\*
&:=\int_\Omega\int_{\Sigma^*}\varepsilon^{-1}_0(\curl \Vv+\curl_y\Vv_1)\cdot (\curl \Vpsi+\curl_y \Vpsi_1)^* +\int_\Omega\int_\Sigma \varepsilon^{-1}_1 \curl_y \Vv_3\cdot \curl_y \Vpsi_3^*\\*
&\quad-k^2\!\!\int_\Omega\int_Y (\Vv+\chi_{\Sigma^*}\nabla_y v_2+\chi_\Sigma \Vv_3)\cdot (\Vpsi+\chi_{\Sigma^*}\nabla_y \psi_2+\chi_{\Sigma}\Vpsi_3)^*\\*
&\quad +\int_{G\setminus \overline{\Omega}}\curl\Vv\cdot \curl\Vpsi^*-k^2\Vv\cdot \Vpsi^*-ik\int_{\partial G} \Vv_T \cdot \Vpsi_T^*.
\end{align*} 
\end{theorem}
The proof is postponed to Section \ref{subsec:homproof}.

We now decouple the influence of the microscale and the macroscale by introducing so called effective parameters.
The macroscopic solution $\Vu$ solves an effective scattering problem, from which we can later on deduce the physically relevant behavior.
We emphasize that $\Vu$ is \emph{not} the weak limit of $\Vu_\delta$.

\begin{theorem}[Cell problems and effective macroscopic problem]
\label{thm:effectiveeq}
The quadruple $(\Vu, \Vu_1, u_2, \Vu_3)$ solves the two-scale equation \eqref{eq:twoscaleeq} if and only if $\Vu\in \VH_{\mathrm{imp}}(G)$ solves the effective macroscopic scattering problem
\begin{equation}
\label{eq:effectiveeq}
\int_G (\varepsilon^{-1})_{\hom} \curl \Vu\cdot \curl\Vpsi^*-k^2\mu_{\hom} \Vu\cdot \Vpsi^*\, dx =\int_{\partial G} \Vg \cdot \Vpsi_T^*\, d\sigma \qquad \forall \Vpsi\in \VH_{\mathrm{imp}}(G)
\end{equation}
and the correctors are 
\begin{equation}
\label{eq:correctors}
\begin{split}
\Vu_1&=\chi_{\Omega}\chi_{\Sigma^*}\sum_j (\curl \Vu)_j\Vw_j^1,\quad
u_2=\chi_\Omega\chi_{\Sigma^*}\sum_j k^2\Vu_j w_j^2 \quad
\text{and} \quad \Vu_3=\chi_\Omega\chi_\Sigma\sum_j k^2\Vu_j\Vw_j^3.
\end{split}
\end{equation}
Here, the homogenized (or effective) material parameters $(\varepsilon^{-1})_{\hom}$ and $\mu_{\hom}$are the identity in $G\setminus\overline{\Omega}$. 
In $\Omega$, they are defined via the solution of cell problems in the following way.

$(\varepsilon^{-1})_{\hom}$ is given as
\begin{align*}
\Bigl((\varepsilon^{-1})_{\hom}\Bigr)_{j,l}&:=\int_{\Sigma^*} \varepsilon^{-1}_0 (\Ve_l+\curl_y \Vw_l^1)\cdot \Ve_j\, dy,
\end{align*}
where $\Vw_l^1\in \widetilde{\VH}_\sharp(\curl, \Sigma^*)$, $l=1,2,3$, solves
\begin{align}
\label{eq:cellproblem1}
\int_{\Sigma^*}\varepsilon^{-1}_0(\Ve_l+\curl_y \Vw_l^1)\cdot \curl_y \Vpsi_1^*&= 0 \qquad \forall \Vpsi_1\in \widetilde{\VH}_{\sharp}(\curl, \Sigma^*).
\end{align}

$\mu_{\hom}$ is given as
\begin{align*}
\Bigl(\mu_{\hom}\Bigr)_{j,l}&:=\int_Y(\Ve_l+k^2\chi_{\Sigma^*}\nabla_y w_l^2+k^2\chi_\Sigma \Vw_l^3)\cdot \Ve_j\, dy,
\end{align*}
where $w_l^2\in H^1_{\sharp, 0}(\Sigma^*)$ and $\Vw_l^3\in \VH_0(\curl, \Sigma)$, $l=1,2,3$, solve
\begin{align}
\label{eq:cellproblem2}
\int_{\Sigma^*}(\Ve_l+k^2\nabla_y w_l^2)\cdot \nabla_y \psi_2^*\, dy &= 0\qquad \forall \psi_2\in H^1_{\sharp,0}(\Sigma^*),\\
\label{eq:cellproblem3}
\int_\Sigma \varepsilon^{-1}_1 \curl_y \Vw_l^3\cdot \curl_y \Vpsi_3^*-k^2\Vw_l^3\cdot \Vpsi_3^*\, dy&=\int_\Sigma \Ve_l\cdot \Vpsi_3^* \qquad\forall \Vpsi_3\in \VH_0(\curl, \Sigma).
\end{align}
\end{theorem}

We emphasize that all cell problems are uniquely solvable due to the Theorem of Lax-Milgram.
(For \eqref{eq:cellproblem3}, note that its left-hand side is coercive because of $\Im(\varepsilon^{-1}_1)<0$.)
Unique solvability of the effective macroscopic equation \eqref{eq:effectiveeq} follows because $\Im(\mu_{\eff})$ is positive-definite in $\Omega$ according to Proposition \ref{prop:effective}, see \cite{BBF15hommaxwell} and \cite[Section 4]{Monk} for details.

The effective macroscopic equation reveals the physical properties of the material: For small $\delta$ it behaves (effectively) like a homogeneous scatterer $\Omega$ with inverse permittivity $(\varepsilon^{-1})_{\hom}$ and permeability $\mu_{\hom}$.
The occurrence of $\mu_{\hom}$, which is not present in \eqref{eq:scatteringweak}, can (physically) be interpreted as artificial magnetism.

\subsection{Comparison with the literature}
\label{subsec:literaturecomp}

In this subsection, we show the equivalence of our results and those available in the literature, namely \cite{CC15hommaxwell} and \cite{BBF09hommaxwell, BBF15hommaxwell}.
However, we already want to emphasize a few new aspects and advantages of our presentation:
\begin{itemize}
\item Presentation of a two-scale equation: This concise and elegant formulation so far has been hidden in the proofs of \cite{CC15hommaxwell}.
\item Uniqueness of the two-scale solution: By a slightly modified definition of the correctors (in comparison to \cite{CC15hommaxwell}, see below), we are able to prove uniqueness in nevertheless simple and natural function spaces. This is clearly a great advantage for analysis. 
\item A new formulation for $\mu_{\hom}$: As already discussed in \cite{BBF15hommaxwell} in detail, the computation of $\mu_{\hom}$ is very challenging, especially with respect to numerical implementations. 
In contrast to the two-dimensional case, $\mu_{\hom}$ does not only depend on the behavior of the magnetic field inside the inclusions (as one might expect), but also the surrounding medium $\Sigma^*$ has to be considered. This, of course, is also persistent in our formulation. Here, however, both parts decouple quite nicely.
Moreover, we are also able to use quite natural and easy to implement function spaces and cell problems in comparison to \cite{BBF15hommaxwell}.
\end{itemize}

\smallskip
\noindent
\textbf{Comparison with \cite{CC15hommaxwell}}.\hspace{0.05in}
Cherednichenko and Cooper \cite[Theorem 2.1]{CC15hommaxwell} obtain a very similar homogenization result to Theorem \ref{thm:twoscaleeq}.
Note that in \cite{CC15hommaxwell}, the sign of the identity term is twisted and a volume source term is present.
Instead of the corrector $\Vu_1$, \cite[Lemma 4.4]{CC15hommaxwell} already includes the effective matrix $(\varepsilon^{-1})_{\hom}$ (named $A_{\hom}$) in the two-scale equation.

The only crucial difference between Theorem \ref{thm:twoscaleeq} and \cite[Theorem 2.1]{CC15hommaxwell} is the different choice or construction of $u_2$ and $\Vu_3$.
Roughly speaking,  our $\Vu_3$ fulfills $\Vu_3 = \nabla_y u^1+u^2$ in $\Sigma$ for the functions $u^1$, $u^2$ defined in \cite[Theorem 2.1]{CC15hommaxwell}.
Basically, we cut off our $u_2$ at the boundary $\partial \Sigma$ and add the ``remaining'' normal boundary traces to $\Vu_3$, whereas in \cite{CC15hommaxwell} the function $u^1$ (corresponding to our $u_2$) is present on the whole cube $Y$.
Moreover, this different definition of the identity correctors leads to the lower regularity $\Vu_3\in \VH_0(\curl, \Sigma)$ instead of $u^2\in H^1_0(\Sigma)$ in \cite{CC15hommaxwell}.
The great advantage of our new formulation is the uniqueness of the two-scale solution.
In \cite{CC15hommaxwell}, only uniqueness of $u$ and of $\nabla_y u^1+u^2$ can be demonstrated.

\smallskip
\noindent
\textbf{Comparison with \cite{BBF15hommaxwell}}.\hspace{0.05in}
Comparing with \cite{BBF15hommaxwell}, we have $(\varepsilon^{-1})_{\hom}=(\varepsilon^{\eff})^{-1}$ and $\mu_{\hom}=\mu^{\eff}$, where $\mu^{\eff}$ and $\varepsilon^{\eff}$ are defined in \cite{BBF15hommaxwell}.
The relationship $(\varepsilon^{-1})_{\hom}=(\varepsilon^{\eff})^{-1}$  is shown in \cite[Lemma 4.4]{CC15hommaxwell}. 
Comparing the definition of $\mu_{\hom}$ and the definition of $\mu^{\eff}$ (via equations (5.23) and (5.21) of \cite{BBF15hommaxwell}), we observe that we have to prove 
\[\chi_{\Sigma^*}\nabla_y w_j^2+\chi_{\Sigma}\Vw_j^3 = \Vu^j,\]
where $w_j^2$ and $\Vw_j^3$ are defined in Theorem \ref{thm:effectiveeq} above and $\Vu^j$ is introduced in \cite[equation (5.21)]{BBF15hommaxwell}.
This means that we have to check that \[\widetilde{\Vw}_j:=\chi_{\Sigma^*}\nabla_y w_j^2+\chi_{\Sigma}\Vw_j^3\in X_0^{\Div}:=\{\Vv\in \VH^1_{\sharp}(Y)|\Div_y \Vv=0 \text{ in }Y, \curl_y \Vv=0 \text{ in }\Sigma, \oint\Vv=0\}\] and that $\widetilde{\Vw}_j$ fulfills equation (5.18) of \cite{BBF15hommaxwell}.
For that, we first prove the following lemma.

\begin{lemma}
\label{lem:identitycorrectors}
Let $w_j^2$ and $\Vw_j^3$ be the solutions to \eqref{eq:cellproblem2} and \eqref{eq:cellproblem3} from Theorem \ref{thm:effectiveeq}. The function $\widetilde{\Vw}_j:=\chi_{\Sigma^*}\nabla_y w_j^2+\chi_\Sigma\Vw_j^3$ fulfills
\[\widetilde{\Vw}_j\in \VH^1_\sharp(Y) \text{ with }\Div_y\widetilde{\Vw}_i=0.\]
Consequently, the same holds true for $\chi_{\Sigma^*}\nabla_y u_2+\chi_\Sigma \Vu_3$ with $u_2, \Vu_3$ the correctors defined in \eqref{eq:correctors} of Theorem \ref{thm:effectiveeq}.
\end{lemma}

\begin{proof}
We have $\Div_y\nabla_y w_j^2=0$ in $\Sigma^*$ because of \eqref{eq:cellproblem2} tested with $\psi_2\in H^1_{\sharp, 0}(\Sigma^*)$ with $\psi_2=0$ on $\partial \Sigma$.
By inserting $\Vpsi_3=\nabla_y \psi_3$ with $\psi_3\in H^1_0(\Sigma)$ into \eqref{eq:cellproblem3}, we obtain $\Div_y \Vw_j^3=0$ in $\Sigma$.
Inserting now test functions as before, but without vanishing (normal) traces on $\partial \Sigma$, we deduce that the normal traces of $\nabla_y w_j^2$ and $\Vw_j^3$ coincide on $\partial \Sigma$.
These properties together imply $\widetilde{\Vw}_j\in \VH_\sharp(\Div, Y)$ with $\Div_y \widetilde{\Vw}_j=0$.
Since also obviously $\widetilde{\Vw}_j\in \VH_\sharp(\curl, Y)$, the assertion follows with \cite[Lemma 4.7]{BBF15hommaxwell}. 
\end{proof}

The inclusion $\widetilde{\Vw}_j\in X_0^{\Div}$ now follows from the previous lemma and because $\widetilde{\Vw}_j$ is given as a gradient on $\Sigma^*$. 
Equation (5.18) of \cite{BBF15hommaxwell} follows from our two cell problems \eqref{eq:cellproblem2} and \eqref{eq:cellproblem3} by a direct calculation, which we do not give here.
Note that in \cite{BBF15hommaxwell}, $\varepsilon_0$ and $\varepsilon_1$ are constants on $\Sigma^*$ and $\Sigma$, respectively.

\section{Stability and regularity analysis for the homogenized system}
\label{sec:analysis}
In the previous section, we have presented two variational problems, the two-scale equation and the homogenized effective system.
This section is devoted to a detailed analysis of those problems with the aim to derive stability and regularity results.
We want to emphasize that this stability and regularity analysis is a prerequisite for the a priori estimates in Section  \ref{subsec:aprioriindef}.

We start this section with two lemmas concerning the two-scale equation \eqref{eq:twoscaleeq}.

\begin{lemma}
\label{lem:equivnorms}
The two-scale energy norm
\begin{equation}
\label{eq:energynorm}
\begin{split}
\|(\Vv, \Vv_1, v_2, \Vv_3)\|^2_e&:=\|\curl\Vv+\curl_y \Vv_1\|^2_{G\times \Sigma^*}+\|\curl_y \Vv_3\|^2_{\Omega\times \Sigma}\\
&\qquad+k^2\|\Vv+\chi_{\Sigma^*}\nabla_y v_2+\chi_\Sigma\Vv_3\|^2_{G\times Y}+k\|\Vv_T\|^2_{\partial G}
\end{split}
\end{equation}
is equivalent to the following (natural) norms on $\CH$
\begin{align}
\nonumber
\|(\Vv, \Vv_1, v_2, \Vv_3)\|^2_\CH:=\|\Vv\|^2_{\VH_{\mathrm{imp}}}+\|\curl_y\Vv_1\|^2_{\Omega\times \Sigma^*}+\|\nabla_y v_2\|^2_{\Omega\times \Sigma^*}+\|\Vv_3\|_{L^2(\Omega; \VH(\curl, \Sigma))},\\
\label{eq:weigthedCHnorm}
\|(\Vv, \Vv_1, v_2, \Vv_3)\|^2_{k; \CH}:=\|\Vv\|^2_{\mathrm{imp}; k}+\|\curl_y\Vv_1\|^2_{\Omega\times \Sigma^*}+k^2\|\nabla_y v_2\|^2_{\Omega\times \Sigma^*}+\|\Vv_3\|_{\curl; k; \Omega\times \Sigma}.
\end{align}
The equivalence constants between \eqref{eq:energynorm} and \eqref{eq:weigthedCHnorm} are independent of $k$.
\end{lemma}

\begin{proof}
The essential ingredient is a sharpened Cauchy-Schwarz inequality for the mixed terms, see the two-dimensional case \cite{OV16hmmhelmholtz}.
Note that due to the choices of $H^1_{\sharp, 0}(\Sigma^*)$ and $\widetilde{\VH}_\sharp(\curl, \Sigma^*)$, the $H^1$- and $\VH(\curl)$-semi norms are norms on those function spaces, respectively.
\end{proof}

The two-scale sesquilinear form $\CB$ from Theorem \ref{thm:twoscaleeq} is obviously continuous with respect to the energy norm \eqref{eq:energynorm} with a $k$-independent constant.
Due to the large kernel of the curl-operator, the $L^2$-term is no compact perturbation of the curl-term.
In order to prove a G{\aa}rding-type inequality, we have to use a Helmholtz-type splitting.
We have the following decomposition of  $(\Vv,\Vv_3) \in \VH_{\mathrm{imp}}(G)\times L^2(\Omega; \VH_0(\curl,\Sigma))$:
\begin{equation}
\label{eq:decomp}
\begin{split}
\Vv+\chi_\Sigma\Vv_3&=\Vz+\chi_\Sigma\Vz_3+\nabla \theta +\chi_\Sigma\nabla_y \theta_3\\
\text{with}\quad \theta&\in H^1_{\partial G}:=\{\phi\in H^1(G)|\phi\text{ constant on }\partial G\},\quad \theta_3\in L^2(\Omega; H^1_0(\Sigma)),\\
\text{and}\quad 0&=(\Vz+\chi_\Sigma\Vz_3, \nabla\eta+\chi_\Sigma\nabla_y \eta_3)_{L^2(G\times Y)}\qquad \forall (\eta, \eta_3)\in H^1_{\partial G}\times L^2(\Omega; H^1_0(\Sigma)).
\end{split}
\end{equation}
The orthogonality in the last line implies a weak divergence-free constraint on $\Vz+\chi_\Sigma\Vz_3$, which implies in turn additional regularity of $\Vz$ and $\Vz_3$, see Remark \ref{rem:reghelmholtzdecomp}.
See \cite{Hipt15maxwellcontdiscr} for a similar approach using the rgeular decomposition.

\begin{lemma}
\label{lem:garding}
Define the sign-flip isomorphism $F:\CH\to\CH$ via 
\[F((\Vv, \Vv_1, v_2, \Vv_3)):=(\Vz-\nabla\theta, \Vv_1, -v_2, \Vz_3-\nabla_y\theta_3)\]
with the Helmholtz decomposition from \eqref{eq:decomp}.
There exist $C_g>0$ and $\gamma_{\mathrm{ell}}>0$, both independent of $k$, such that
\begin{equation}
\label{eq:gardingtwosc}
\bigl|\CB((\Vv, \Vv_1, v_2, \Vv_3), F((\Vv, \Vv_1, v_2, \Vv_3)))+C_gk^2\|\Vz+\chi_\Sigma\Vz_3\|^2_{L^2(G\times Y)}\bigr|\geq \gamma_{\mathrm{ell}}\|\Vv\|^2_e.
\end{equation}
\end{lemma}
\begin{proof}
The sign-flip isomorphism and the added identity term correct the ``wrong'' sign of the sesquilinear form $\CB$ and make it coercive.
Mixed terms between $\theta$ and $\Vz$ or $\theta_3$ and $\Vz_3$, respectively, either vanish due to the orthogonality of the Helmholtz decomposition or can be absorbed using Cauchy-Schwarz and Young inequality.
\end{proof}

We now analyze the stability and higher regularity of the two-scale solution by analyzing the cell problems and the homogenized equation separately.
As we have already discussed, all cell problems are coercive, so that their stability is an easy consequence.

\begin{lemma}
\label{lem:stabcell}
The correctors fulfill the stability estimates
\begin{align*}
\|\curl_y\Vu_1\|_{L^2(\Omega\times \Sigma^*)}&\lesssim\|\curl \Vu\|_{L^2(\Omega)},&&&
k\|\nabla_y u_2\|_{L^2(\Omega\times \Sigma^*)}&\lesssim \|\Vu\|_{\mathrm{imp}; k; \Omega},\\
\|\curl_y \Vu_3\|_{L^2(\Omega\times \Sigma)}+k\|\Vu_3\|_{L^2(\Omega\times \Sigma)}&\lesssim \|\Vu\|_{\mathrm{imp}; k; \Omega}.
\end{align*} 
\end{lemma}

With this knowledge on the cell problems, we can now deduce some useful properties of the effective parameters.

\begin{proposition}
\label{prop:effective}
The effective parameters have the following properties:
\begin{itemize}
\item $(\varepsilon^{-1})_{\hom}$ is a piece-wise constant, real-valued, symmetric positive definite matrix;
\item $\mu_{\hom}$ is a piece-wise constant, complex-valued, symmetric (not hermitian!) matrix with upper bound independent from $k$;
\item $\Im(\mu_{\hom})$ is symmetric positive-definite (and thus $\mu_{\hom}$ is invertible);
\item if $\varepsilon_1$ is constant, we have
\[\Im(\mu_{\hom})\xi\cdot \xi^*\geq Ck^{-2}|\xi|^2\qquad \forall \xi\in \cz^3.\]
\end{itemize}
\end{proposition}

\begin{proof}
The characterization of $(\varepsilon^{-1})_{\hom}$ is well-known and follows from the ellipticity of the corresponding cell problem, see \cite{HOV15maxwellHMM}.

The upper bound on $\mu_{\hom}$ easily follows from the stability bounds on $u_2$ and $\Vu_3$ given in the previous lemma.
For the positive-definiteness of $\Im(\mu_{\hom})$ we deduce from the cell problems that
\[\Im(\mu_{\hom}\xi\cdot \xi^*)=k^2\int_\Sigma \Im((\varepsilon^{-1}_1)^*)|\curl_y \Vw^3_\xi|^2\, dy,\]
where $\Vw_\xi$ is the solution to cell problem \eqref{eq:cellproblem3} with right-hand side $\xi$.
Note that by assumption it holds $\Im((\varepsilon^{-1}_1)^*)>0$.
$\curl_y \Vw^3_\xi=0$ is only possible if $\xi=0$ due to the cell problem and its boundary condition.

For the case of constant $\varepsilon_1$, we use the equivalence to the effective $\mu$ given in \cite{BBF15hommaxwell} (cf.\ Subsection \ref{subsec:literaturecomp}).
Then, we can use the following representation, which is equation (6.16) of \cite{BBF15hommaxwell},
\[(\mu_{\hom})_{j,l}=\mbox{Id}_{jl}+\sum_n\frac{\varepsilon_1k^2}{\lambda_n-\varepsilon_1k^2}\Bigl(\int_Y\Vphi_n\cdot \Ve_j\Bigr)\Bigl(\int_y \Vphi_n\cdot \Ve_l\Bigr).\]
Here, $(\Vphi, \lambda_n)$ are eigenfunctions and eigenvalues of a vector-Laplacian on $Y$.
Now, the lower bound can be shown as in the two-dimensional case in \cite{OV16hmmhelmholtz}.
\end{proof}

The regularity results for the cell problems can be deduced from well-known regularity theory, see \cite{Hipt02FEem} for details.

\begin{proposition}
\label{prop:regcellandcavity}
There are $1/2<t_j\leq 1, j=1,2,3$ such that $\Vu_1\in L^2(\Omega; \VH^{t_1}(\curl, \Sigma^*))$, $u_2\in L^2(\Omega; H^{1+t_2}(\Sigma^*))$ and $\Vu_3\in L^2(\Omega; \VH^{t_3}(\Sigma))$ with the regularity estimates
\begin{align*}
\|\curl_y\Vu_1\|_{L^2(\Omega; \VH^{t_1}(\Sigma^*))}&\lesssim \|\Vu\|_{\curl; k; G}\\
k\|u_2\|_{L^2(\Omega; H^{1+t_2}(\Sigma^*))}&\lesssim  \|\Vu\|_{\mathrm{imp};k; G}\\
\|\curl\Vu\|_{L^2(\Omega; \VH^{t_3}(\curl, \Sigma))}+k\|\Vu_3\|_{L^2(\Omega; \VH^{t_3}(\Sigma))}&\lesssim (1+k)\|\Vu\|_{\mathrm{imp}; k; G}.
\end{align*}
We have $t_j=1$ for all $j$ if $\Sigma$ is of class $C^2$.
\end{proposition}

The higher regularity for the effective scattering equation is more difficult to derive due to the impedance boundary condition.
As the effective parameters $(\varepsilon^{-1})_{\hom}$ and $\mu_{\hom}$ are piecewise constant, we can only expect piecewise higher regularity.
Therefore, we introduce $\VH^s_{pw}(\curl, G)=\VH(\curl, G)\cap \VH^s(\curl, \Omega)\cap \VH^s(\curl, G\setminus \overline{\Omega})$ with the corresponding norm.
For the definition of the trace spaces, we use the notation of \cite{GM12maxwellimpedance} and refer to \cite{BC01tracescurldiv, BCS02tracescurl, Moiola} for details on the spaces.

\begin{proposition}
\label{prop:regscattering}
Let $\Vf\in \VH(\Div, G)$ with $\Div \Vf=0$. 
Let $\Vu$ be the solution to \eqref{eq:effectiveeq} with additional volume term $\Vf$ on the right-hand side.
\begin{itemize}
\item If $\Omega$ and $G$ have $C^2$-boundary and $\Vg\in \VH^{1/2}_T(\partial G)$, then $\Vu\in \VH^1_{pw}(\curl, G)$.
\item If $G$ is convex and $\Vg\in \VH^{s_g}_T(\partial G)$ for $0<s_g<1/2$, there is $1/2<s\leq 1/2+s_g$, only depending on the shape of $\Omega$ and $G$, such that $\Vu\in \VH^{s}_{pw}(\curl, G)$.
\end{itemize}
In both cases, we have the regularity estimate
\[\|\curl\Vu\|_{\VH^s_{pw}(G)}+k\|\Vu\|_{\VH^s_{pw}(G)}\leq C\bigl((1+k)\|\Vu\|_{\curl; k; G}+\|\Vf\|_{L^2(G)}+\|\Vg\|_{\VH^{s_g}(\partial G)}\bigr).\]

Moreover, if $\Vu\in \VH^s_{pw}(G)$ with $1/2<s\leq 1$, we also have $\Vu\in \VH^{s-1/2}_{\parallel}(\partial G)\cap\VH(\curl_{\partial G})$ with
\begin{equation}
\label{eq:regbdryscatter}
\begin{split}
k^{1/2}(\|\Vu\|_{\VH^{s-1/2}_\parallel(\partial G)}+\|\curl_{\partial G}(\Vu_T)\|_{L^2(\partial G)})&\leq Ck^{1/2} \|\Vu\|_{\VH^s_{pw}(G)}.
%\\&\leq C \bigl(k^{1/2}\|\Vu\|_{\curl; k; G}+\|\Vf\|_{L^2(G)}+\|\Vg\|_{\VH^{s_g}(\partial G)}\bigr).
\end{split}
\end{equation}
\end{proposition}

\begin{proof}
The proof can be easily adopted from the case of scalar-valued constant material parameters in \cite{Moiola}.
We refer to \cite{BGL13regularitymaxwell, CDN99maxwellinterface} for  other results on higher regularity of curl-curl-problems with piece-wise constant coefficients.
The regularity on the boundary directly follows from the continuity of trace operators, see \cite{BC01tracescurldiv, BCS02tracescurl, BH03bdryfem}.
\end{proof}

\begin{remark}
\label{rem:reghelmholtzdecomp}
The arguments from Propositions \ref{prop:regcellandcavity} and \ref{prop:regscattering} can also be employed to show higher regularity for the Helmholtz  decomposition \eqref{eq:decomp}: We have $\Vz\in \VH^s_{pw}(G)$ and $\Vz_3\in L^2(\Omega, \VH^{t_3}(\curl, \Sigma))$.
\end{remark}

In order to have a full regularity estimate only in terms of the data, we need a stability result, i.e.\ the dependence of the solution in its natural norm (here $\|\cdot\|_{\imp; k; G}$) on the data.
Fredholm theory gives us such a stability result, but without explicit  dependence of the constant on $k$.
We now assume an explicit, polynomial stability constant.

\begin{assumption}
\label{ass:polstable}
We assume that the solution to the homogenized macroscopic equation  \eqref{eq:effectiveeq} with additional volume term $\Vf\in \VH(\Div, G)$ with $\Div\Vf=0$ is polynomially stable, i.e.\ the unique solution fulfills for some $q\in \nz_0$ and an $k$-independent constant $C_{\stab}$
\begin{equation}
\label{eq:polstable}
\|\Vu\|_{\imp; k, G}\leq C_{\stab}\, k^q(\|\Vf\|_{L^2(G)}+\|\Vg\|_{L^2(\partial G)}).
\end{equation}
\end{assumption}

The only polynomial stability results for time-harmonic Maxwell equations available in the literature so far consider the case of constant coefficients, see \cite{FW14dgmaxwell, HMP11stabilityMaxwell, Moiola}.
The setting of the effective homogenized equation \eqref{eq:effectiveeq} exhibits new challenges for the stability analysis:  discontinuous, namely piece-wise constant,  and matrix-valued coefficients and a partly complex parameter $\mu$.
In order to cope with these challenges, we first generalize the known results to the class of real- and matrix-valued, Lipschitz continuous coefficients.
More precisely, we have the following proposition, which is proved in Section \ref{subsec:stabilproof}.

\begin{proposition}
\label{prop:stabilitylipschitz}
Assume that there is $\gamma>0$ such that
\begin{align}
\label{eq:geometry}
x\cdot \Vn_G\geq \gamma \text{ on }\partial G &&x\cdot \Vn_\Omega\geq 0\text{ on }\partial \Omega,
\end{align}
where $\Vn$ denotes the outer normal of the domain specified in the subscript.
Let $\Vv\in\VH_{\mathrm{imp}}(G)$ be the unique solution to 
\begin{equation}
\label{eq:scatteringweakstabil}
\int_G \!A\curl\Vv\cdot \curl\Vpsi^*-k^2B\Vv\cdot \Vpsi^*\, dx -ik\int_{\partial G}\!\!\beta\Vv_T\cdot \Vpsi^*_T\, d\sigma = \int_G\!\Vf\cdot \Vpsi^*\, dx+\int_{\partial G}\!\Vg\cdot \Vpsi^*_T\, d\sigma \quad \forall \Vpsi\in \VH_{\imp}(G)
\end{equation}
with $\Vf\in \VH(\Div, G)$ with $\Div\Vf=0$, $\Vg\in L^2_T(\partial G)$
and $A, B\in W^{1, \infty}(G)$ fulfilling the assumptions
\begin{itemize}
\item $A, B$ are real-valued symmetric positive-definite
\item $A=\alpha(x)\Id$, $B=\beta(x)\Id$ in a neighborhood of the boundary $\partial G$
\item the matrix $DA\cdot \Vx$ is negative semi-definite and $DB\cdot \Vx$ is positive semi-definite, where $(DA\cdot \Vx)_{jl}:=\sum_n\partial_nA_{j,l}x_l$.
\end{itemize}
There exists a constant $C>0$, depending only on $G$, $k_0$, and the upper and lower bounds (eigenvalues) of $A$ and $B$, but not on $k$, the data $\Vf$ and $\Vg$, or any derivative information of $A$ and $B$, such that
\begin{equation}
\label{eq:stability}
\|\Vv\|_{\mathrm{imp}, k, G}\leq C (\|\Vf\|_{L^2(G)}+\|\Vg\|_{L^2(\partial G)}).
\end{equation}
\end{proposition}

The geometrical assumption \eqref{eq:geometry} is the common assumption for scattering problems, see \cite{Moiola, HMP11stabilityMaxwell}. 
It can, for example, be fulfilled if $\Omega$ is convex (and w.l.o.g.\ $0\in \Omega$) and $G$ is chosen appropriately.
Note that the conditions on the derivatives of the coefficients are similar to those for the Helmholtz equation, see \cite{OV16hmmhelmholtz} and the remarks therein.
We emphasize that we obtain the same stability result, i.e.\ $q=0$, as for Maxwell's equations with constant coefficients, see \cite{HMP11stabilityMaxwell, Moiola}.
This generalization to a wider class of coefficients maybe of interest on its own.

We can now prove Assumption \ref{ass:polstable} with $q=3$ for the setting of the homogenized equation \eqref{eq:effectiveeq}.
More precisely, we have the following theorem, which is proved in Section \ref{subsec:stabilproof}.

\begin{theorem}
\label{thm:stabilityeff}
Let $G$ and $\Omega$ fulfill \eqref{eq:geometry}.
Furthermore assume that $(\varepsilon^{-1})_{\hom}|_{G\setminus \overline{\Omega}}-(\varepsilon^{-1})_{\hom}|_\Omega$ is negative semi-definite. 
We assume that $\Im(\mu_{\hom})\geq k^{-2}$, see Proposition \ref{prop:effective} for constant $\varepsilon_1$.
Let $\Vu$ be the solution to \eqref{eq:effectiveeq} with additional volume term $\int_G \Vf\cdot\Vpsi^*\, dx$ on the right hand-side for $\Vf\in \VH(\Div, G)$ with $\Div\Vf=0$. 
Then there is $C_{\stab,0}$ only depending on the geometry, the parameters, and $k_0$, such that $\Vu$ satisfies the stability estimate
\[ \|\Vu\|_{\mathrm{imp},k,G}\leq C_{\stab, 0}(k^3\|\Vf\|_{L^2(\Omega)}+k^2\|\Vf\|_{L^2(G\setminus \overline{\Omega})}+k^{3/2}\|\Vg\|_{L^2(\partial G)}+k^{-1}\|\Vg\|_{\VH^{s_g}(\partial G)}).\]
\end{theorem}
 
The assumption on $(\varepsilon^{-1})_{\hom}$ in fact is an assumption on $\varepsilon^{-1}_0$ and can be fulfilled for appropriate choices of material inside and outside the scatterer.
It comes from the conditions on the derivative of $A$ in Proposition \ref{prop:stabilitylipschitz} and is similar to the two-dimensional case in \cite{OV16hmmhelmholtz}.
The different powers in $k$ in comparison to Proposition \ref{prop:stabilitylipschitz} are caused by the complex-valued $\mu_{\eff}$ and the dependence of $\Im(\mu_{\eff})$ on $k$, see also the discussion in Section \ref{subsec:stabilproof}.
Note that we obtain the same powers in $k$ as in the two-dimensional stability estimate in \cite{OV16hmmhelmholtz}.

In the following, we will work with the (abstract) polynomial stability of Assumption \ref{ass:polstable} and keep in mind that we have obtained an explicit (maximal) $q$ in Theorem \ref{thm:stabilityeff}.
Hence, we can conclude that the regularity constant from Proposition \ref{prop:regscattering} behaves like $k^{q+1}$.
Furthermore, we can also deduce the following form for the inf-sup-constant.
\begin{lemma}
\label{lem:infsupconst}
Under Assumption \ref{ass:polstable}, the sesquilinear form $\CB$ is inf-sup-stable with
\[\inf_{\underline{\Vv}\in\CH}\sup_{\underline{\Vw}\in\CH}\frac{|\CB(\underline{\Vv}, \underline{\Vw})|}{\|\underline{\Vv}\|_e\, \|\underline{\Vw}\|_e}\geq \frac{\gamma_{\mathrm{ell}}}{1+C_{\stab, e}C_gk^{q+1}},\]
where $C_{\stab, e}$ is the stability constant for the two-scale problem and consists of $C_{\stab, 0}$ from Assumption \ref{ass:polstable} and the stability constants from Lemma \ref{lem:stabcell} (which are all $k$-independent).
\end{lemma}

\begin{proof}
Let $\underline{\Vv}=(\Vv, \Vv_1, v_2, \Vv_3)\in \CH$ be arbitrary and let $\underline{\Vw}\in \CH$ be the solution to the adjoint two-scale problem with right hand-side $C_gk^2(\Vz+\chi_\Sigma\Vz_3)$ for the Helmholtz decomposition of $\underline{\Vv}$ according to \eqref{eq:decomp}.
Note that $\Vz$ and $\Vz_3$ are divergence-free and therefore, Assumption \ref{ass:polstable} can be applied.
Recall the sign-flip isomorphism and the G{\aa}rding inequality from Lemma \ref{lem:garding}.
On the one hand, we have
\begin{align*}
\bigl|\CB(\underline{\Vv}, F(\underline{\Vv})+\underline{\Vw})\bigr|&=\bigl|\CB(\underline{\Vv}, F(\underline{\Vv}))+C_gk^2(\Vz+\chi_\Sigma\Vz_3, \Vv+\chi_{\Sigma^*}\nabla_y v_2+\chi_{\Sigma}\Vv_3)\bigr|\\
&=\bigl|\CB(\underline{\Vv}, F(\underline{\Vv}))+C_gk^2\|\Vz+\chi_\Sigma\Vz_3\|^2\bigr|\geq \gamma_{\mathrm{ell}}\|\underline{\Vv}\|_e^2.
\end{align*}
On the other hand, it holds that
\begin{align*}
\|F(\underline{\Vv})+\underline{\Vw}\|_e&\leq \|F(\underline{\Vv})\|_e+\|\Vw\|_e\leq \|\underline{\Vv}\|_e+C_{\stab, e}k^q C_gk^2\|\Vz+\chi_\Sigma\Vz_3\|_{L^2(G\times Y)}\\
&\leq (1+C_{\stab, e}C_gk^{q+1})\|\underline{\Vv}\|_e.
\end{align*}
Combining both estimates finishes the proof.
\end{proof}

\section{Numerical method and error analysis}
\label{sec:numerics}
As explained in the introduction, a direct discretization of the heterogeneous problem \eqref{eq:scatteringweak} is infeasible due to the necessary small mesh width. 
In Subsection \ref{subsec:HMM}, we introduce the HMM and perform its rigorous numerical analysis in Subsection \ref{subsec:aprioriindef}.

\subsection{The Heterogeneous Multiscale Method}
\label{subsec:HMM}
The idea of the Heterogeneous Multiscale Method (HMM) is to imitate the homogenization procedure and thereby provide a method with $\delta$-independent mesh sizes.
Following the original idea \cite{Ohl05HMM} for elliptic diffusion problems, we concentrate on the direct discretization of the two-scale equation \eqref{eq:twoscaleeq}.
This point of view is vital for the numerical analysis in Subsection \ref{subsec:aprioriindef}.
However, we will also shortly explain below how this direct discretization can be decoupled into coarse- and fine-scale computations in the traditional fashion of the HMM as presented in \cite{EE03hmm, EE05hmm}.

In this and the next section, we assume that $\Sigma$, $\Omega$, and $G$ are Lipschitz polyhedra (in contrast to the $C^2$ boundaries in the analytic sections).
The reason is that the $C^2$ boundaries can be approximated by a series of more and more fitting polygonal boundaries. 
This procedure of boundary approximation results in non-conforming methods, i.e.\ the discrete function spaces are no subspaces of the analytic ones.
We avoid this difficulty in our numerical analysis by assuming polygonally bounded domains by now. The new assumption reduces the possible higher regularity of solutions as discussed in Section \ref{sec:analysis}. However, we can always obtain the maximal regularity in the limit of polygonal approximation of $C^2$ boundaries, which we have in mind as application case.

Denote by $\CT_H=\{ T_j|j\in J\}$ and $\CT_h=\{S_l|l\in I\}$ conforming and shape regular triangulations of $G$ and $Y$, respectively. Additionally, we assume that $\CT_H$ resolves the partition into $\Omega$ and $G\setminus \overline{\Omega}$ and that $\CT_h$ resolves the partition of $Y$ into $\Sigma$ and $\Sigma^*$ and is periodic in the sense that it can be wrapped to a regular triangulation of the torus (without hanging nodes). 
We define the local mesh sizes $H_j:=\diam(T_j)$ and $h_l:=\diam(S_l)$ and the global mesh sizes $H:=\max_{j\in J}H_j$ and $h:=\max_{l\in I}h_l$.
We denote the barycenters by $x_j\in T_j$ and $y_l\in S_l$.

We use the following conforming finite element spaces, associated with the meshes $\CT_H$ or $\CT_h$,
\begin{itemize}
\item the classical linear Lagrange elements $\widetilde{W}_h(\Sigma^*)\subset H^1_{\sharp, 0}(\Sigma^*)$ (adopted to periodic boundary conditions and zero mean value);
\item N{\'e}d{\'e}lec edge elements of lowest order $\VV_H\subset \VH_{\mbox{\tiny{imp}}}(G)$, $\VV_h(\Sigma)\subset \VH_0(\curl, \Sigma)$, and $\widetilde{\VV}_h(\Sigma^*)\subset \widetilde{\VH}_\sharp(\curl, \Sigma^*)$.
\end{itemize}
The space $\widetilde{\VV}_h(\Sigma^*)$ is used to discretize the first corrector $\Vu_1$.
As discussed in Section \ref{sec:homogenization}, we are only interested in its curl. However, in order to obtain a unique solution $\Vu_{h,1}$, we have to apply a suitable stabilization procedure to the corresponding cell problem, such as a Lagrange multiplier or weighted divergence regularization, see \cite{CD00maxwellsingularities, CD02weightedregulmaxwell}.
As an alternative, we can also directly discretize $\curl_y \Vu_1(x, \cdot)$ in a suitable finite element space.

\begin{definition}
\label{def:hmmdiscrtwosc}
Define the piecewise constant approximations $\varepsilon_{0, h}^{-1}$ and $\varepsilon_{1,h}^{-1}$ on $\Omega \times Y$ by $\varepsilon^{-1}_{\cdot, h}(x,y)|_{T_j\times S_l}:=\varepsilon_{\cdot}^{-1}(x_j, y_l)$.
The discrete two-scale solution
\[(\Vu_H, \Vu_{h,1}, u_{h,2}, \Vu_{h,3})\in  \VV_{H,h}:=\VV_H\times L^2(\Omega; \widetilde{\VV}_h(\Sigma^*))\times L^2(\Omega; \widetilde{W}_h(\Sigma^*))\times L^2(\Omega;\VV_h(\Sigma))\]
is defined as the solution of
\begin{equation}
\label{eq:discretetwosceq}
\begin{split}
&\!\!\!\!\CB_h((\Vu_H, \Vu_{h,1}, u_{h,2}, \Vu_{h,3}), (\Vpsi_H, \Vpsi_{h,1}, \psi_{h,2}, \Vpsi_{h,3}))=(\Vg,(\Vpsi_H)_T)_{\partial G} \\
&\qquad \forall (\Vpsi_H, \Vpsi_{h,1}, \psi_{h,2}, \Vpsi_{h,3})\in \VV_H\times L^2(\Omega; \widetilde{\VV}_h(\Sigma^*))\times L^2(\Omega; \widetilde{W}_h(\Sigma^*))\times L^2(\Omega; \VV_h(\Sigma)),
\end{split}
\end{equation}
where the sesquilinear form $\CB_h$ equals $\CB$ from Theorem \ref{thm:twoscaleeq}, but with the coefficients $\varepsilon^{-1}_\cdot$ replaced by the piecewise constant approximations $\varepsilon^{-1}_{\cdot, h}$.
\end{definition}
In order to evaluate the integrals over $G$ in $\CB_h$, one introduces quadrature rules, which are exact for the given ansatz and test spaces.
In our case of piecewise linear functions, it suffices to choose the one-point rule $\{|T_j|, x_j\}$ with the barycenter $x_j$ for the curl part and a second order quadrature rule $Q^{(2)}:=\{q_l, x_l\}_l$ with $l=1, \ldots, 4$ for the identity part on each tetrahedron.
As a consequence, the functions $\Vu_{h,1}$, $u_{h,2}$, and $\Vu_{h, 3}$ will also be discretized in their part depending on the macroscopic variable $x$:
In fact, one has $\Vu_{h,1}\in S_H^0(\Omega;\widetilde{\VV}_h(\Sigma^*))$, $u_{h, 2}\in S_H^1(\Omega; \widetilde{W}_h^1(\Sigma^*))$, and $\Vu_{h,3}\in S_H^1(\Omega; \VV_h(\Sigma))$.
Here, the space of discontinuous, piecewise $p$-polynomial (w.r.t.\ $x$) discrete functions is defined as
\begin{align*}
S_H^p(\Omega; X_h)&:=\{v_h\in L^2(\Omega; X)|\,v_h(\cdot, y)|_{T_j}\in\pz^p\; \forall j\in J, y\in Y; v_h(x, \cdot)\in X_h\;\forall x\in \Omega\},
\end{align*}
for any conforming finite element space $X_h\subset X$.
Note that $u_{h,2}$ and $\Vu_{h, 3}$ are piecewise $x$-linear discrete functions, since $Q^{(2)}$ consists of $4$ quadrature points on each tetrahedron.

The functions $\Vu_{h,1}$, $u_{h,2}$, and $\Vu_{h,3}$ are the discrete counterparts of the analytical correctors $\Vu_1$, $u_2$ and $\Vu_3$ introduced in Theorem \ref{thm:effectiveeq}.
These corrections are an important part of the HMM-approximation and cannot be neglected as higher order terms: 
For Maxwell's equations, we saw in \cite{HOV15maxwellHMM, GHV17lodmaxwell} that $u_{h,2}$ is necessary to obtain good $L^2$ approximations.
Additionally, the corrector $\Vu_{h,3}$ encodes the behavior of the solution inside the inclusions, see \cite{OV16hmmhelmholtz} for the Helmholtz equation with high contrast.

$\Vu_{h,1}$, $u_{h,2}$, and $\Vu_{h,3}$ are correctors to the macroscopic discrete function $\Vu_H$ and solve discretized cell problems.
These cell problems, posed on the unit cube $Y$, can be transferred back to $\delta$-scaled and shifted unit cubes $Y_j^\delta=x_j+\delta Y$, where $x_j$ is a  macroscopic quadrature point. 
This finally gives an equivalent formulation of \eqref{eq:discretetwosceq} in the form of a (traditional) HMM.
The formulation using a macroscopic sesquilinear form with local cell reconstructions is used in practical implementations.
We emphasize that the presented HMM also works for locally periodic $\varepsilon^{-1}_0$ and $\varepsilon^{-1}_1$ depending on $x$ and $y$.

\subsection{A priori error estimates}
\label{subsec:aprioriindef}
Based on the definition of the HMM as direct discretization of the two-scale equation (Definition \ref{def:hmmdiscrtwosc}), we analyze its well-posedness and quasi-optimality in Theorem \ref{thm:infsupcond}. 
This quasi-optimality is a kind of C{\'e}a lemma for indefinite problems and leads to explicit a priori estimates in Corollary \ref{cor:errorestimates} and Theorem \ref{thm:dualestimates}.
As discussed for the G{\aa}rding inequality and in general in \cite{GHV17lodmaxwell}, we will again frequently use the Helmholtz decomposition in our analysis.

For simplicity, we consider the case of constant $\varepsilon_0$ and $\varepsilon_1$ here, so that $\CB_h=\CB$.
The non-conformity occurring from numerical quadrature only leads to additional data approximation errors, which are of higher order for sufficiently smooth coefficients (e.g.\ Lipschitz continuous).
Let us define the error terms $e_0=\Vu-\Vu_H$, $e_1=\Vu_1-\Vu_{h,1}$, $e_2=u_2-u_{h,2}$, and $e_3=\Vu_3-\Vu_{h,3}$ and set $\Ve:=(e_0, e_1, e_2, e_3)$.
We will only estimate these errors and leave the modeling error, introduced by homogenization, apart.
All proofs are postponed to Subsection \ref{subsec:hmmproof}.

\begin{theorem}[Discrete inf-sup-condition and quasi-optimality]
\label{thm:infsupcond}
Under the resolution condition 
\begin{equation}
\label{eq:resolcond1}
C_cC_{\appr}(C_g+2)(k^{q+2}(H^s+h^{t_1}+h^{t_2}+h^{t_3})+ k^{q+3/2}H^{s-1/2})\leq \gamma_{\mathrm{ell}}/2,
\end{equation}
we have the discrete inf-sup condition
\[\inf_{\Vv_{H,h}\in \VV_{H,h}}\sup_{\Vw_{H,h}\in \VV_{H,h}}\frac{|\CB(\Vv_{H,h}, \Vw_{H,h})|}{\|\Vv_{H,h}\|_e\|\Vw_{H,h}\|_e}\geq \frac{\gamma_{\mathrm{ell}}}{2+\gamma_{\mathrm{ell}}/C_c+2C_gC_{\stab, e}k^{q+1}} \sim k^{-(q+1)}\] 
and the error between the analytical and discrete two-scale solution satisfies
\begin{equation}
\label{eq:quasiopt}
\|(e_0, e_1, e_2, e_3)\|_e\leq \frac{2C_c}{\gamma_{\mathrm{ell}}}\inf_{\Vv_{H,h}\in \VV_{H,h}}\|\underline{\Vu}-\Vv_{H,h}\|_e.
\end{equation}
\end{theorem}

The approximation result of Lemma \ref{lem:adjapprox} (see below) gives explicit convergences rates from the quasi-optimality.
\begin{corollary}
\label{cor:errorestimates}
Under the assumptions of Theorem \ref{thm:infsupcond}, the energy error can be estimated as
\begin{equation*}
\begin{split}
\|(e_0, e_1, e_2, e_3)\|_e
&\lesssim (k^{q+1}(H^s + h^{t_1} + h^{t_2} + h^{t_3})+ k^{q+1/2}H^{s-1/2})\|\Vg\|_{L^2(\partial G)}.
\end{split}
\end{equation*}
\end{corollary}
Assuming smooth domains (i.e.\ maximal regularity), the a priori estimate gives linear convergence for the volume terms and $H^{1/2}$ convergence rate for the boundary terms.
These are classical optimal convergence rates under mesh refinement for problems posed in $\VH(\curl)$, see \cite{EG17maxwell, GM12maxwellimpedance}.

As discussed in \cite{HOV15maxwellHMM, GHV17lodmaxwell}, we have to go to dual norms to obtain higher order convergence.

\begin{theorem}
\label{thm:dualestimates}
Let $e_0+\chi_\Sigma e_3=\Vz+\chi_\Sigma\Vz_3+\nabla \theta+\chi	_\Sigma\nabla_y \theta_3$ be the Helmholtz decomposition of the error according to \eqref{eq:decomp}.
This decomposition satisfies the following a priori estimate
\begin{align*}
\|\Vz+\chi_\Sigma\Vz_3\|_{L^2(G\times Y)}+\|\theta+\chi_\Sigma\theta_3\|_{L^2(G\times Y)}\lesssim (k^{q+1}(H^s+h^{t_1}+h^{t_2}+ h^{t_3})+k^{q+1/2} H^{s-1/2})\|\Ve\|_e.
\end{align*}
\end{theorem}

\smallskip
Assuming maximal regularity, i.e.\ $s=t_1=t_2=t_3=1$, and optimal stability with $q=0$, the resolution condition reads $k^2(H+h)+k^{3/2}H^{1/2}\lesssim 1$.
The first part $k^2(H+h)$ comes from the volume terms and is unavoidable for the Helmholtz equation, see \cite{OV16hmmhelmholtz} and \cite{Sau06convanahelmholtz}.
The second part $k^{3/2}H^{1/2}$ is caused by the boundary terms,  which are an essential part of the energy norm for Maxwell equations.
In contrast to the Helmholtz equation, they cannot be estimated against the volume terms by using a trace inequality and thus, seem to be unavoidable as well.
The powers in $k$ and $H$ for the resolution condition caused by the boundary terms is consistent with the volume terms: for both, $k$ and $H$, the power is reduced by $1/2$.
Unfortunately, despite this consistency, the part $k^{3/2}H^{1/2}$ is  the dominating part in the resolution condition and finally, leads to a condition like ``$k^3 H$ small''.

We emphasize that it is natural that $h$ enters the resolution condition because the third cell problem depends on $k$.
Note that $h$ denotes the mesh width of the unit square and is independent from $\delta$.
Our explicit stability estimate in Theorem \ref{thm:stabilityeff} yields $q=3$ and thus, a kind of ``worst case'' resolution condition: It is certainly sufficient for well-posedness and quasi-optimality, but may well be sub-optimal for most frequencies $k$, since in particular the influence from $\Im(\mu_{\hom})$ may be overestimated.
This has been discussed in detail and examined in the numerical experiment for the Helmholtz equation in \cite{OV16hmmhelmholtz}.
We emphasize that the resolution condition can be improved if better stability results are known, which is outside the scope of this work.
Moreover, we underline that previous works \cite{GM12maxwellimpedance, Hipt02FEem, Hipt15maxwellcontdiscr, Monk} so far have only proved well-posedness for sufficiently fine meshes without explicit $k$-dependent resolution condition.

Furthermore, we note that the resolution condition may be reduced, which has been extensively studied for the Helmholtz equation.
For Maxwell's equations, developments in that direction include (hybridizable) discontinuous Galerkin methods \cite{FW14dgmaxwell, FLX16hdgmaxwell, LCQ17hdgmaxwell} or (plane wave) Trefftz methods \cite{HMP13trefftzmaxwell}, just to name a few.
Also the Localized Orthogonal Decomposition (LOD) \cite{MP14LOD, Pet15LODreview} has shown promising results for the Helmholtz equation in \cite{GP15scatteringPG, P17LODhelmholtz}.
Only recently, it has been discussed for elliptic $\VH(\curl)$-problems \cite{GHV17lodmaxwell}.
The definition of the HMM as direct diescretization of the two-scale equation makes an additional application of the LOD possible, see \cite{OV16a} for Helmholtz-type problems.

As already remarked in \cite{Ohl05HMM, HOV15maxwellHMM, OV16hmmhelmholtz}, the definition of the HMM as direct discretization of the two-scale equation is the crucial starting point for the proofs of the a priori error estimates.
In particular, it also enables the derivation of a posteriori error estimates.

\section{Main proofs}
\label{sec:proofs}
In this section all essential proofs on the two-scale equation, the stability of the homogenized equation and the numerical analysis of the HMM are given.

\subsection{Proof of the two-scale equation}
\label{subsec:homproof}
In this section, we show the two-scale equation \eqref{eq:twoscaleeq}.
It closely follows \cite{BBF15hommaxwell} and mainly differs in the form of the two-scale convergence, so that we will focus on that part.

\begin{proof}[Proof of Theorem \ref{thm:twoscaleeq}]
\emph{First step: A priori bounds. }
Assume that $\Vu_\delta$ is uniformly bounded in $L^2(G)$.
We then easily deduce that $\sqrt{|\varepsilon^{-1}_\delta|}\curl\Vu_\delta$ is also uniformly bounded in $L^2(G)$.

\emph{Second step: two-scale convergences. }
By the a priori bounds, $\Vu_\delta$ converges weakly in $\VH(\curl, G\setminus\overline{\Omega})$ to some $\Vu$.
Using \cite[Prop.\ 7.1]{BBF15hommaxwell}, we deduce $\Vu\in \VH(\curl, G\setminus\overline{\Omega})$.
Since $G\setminus \overline{\Sigma}_\delta$ is a simply connected domain, the two-scale convergences from Wellander et al.\ \cite{Well2, Well3} and Visintin \cite{Visintin} can be applied (formally with the help of extension by zero in $\Sigma_\delta$):
There exist $\Vu\in \VH_{\mbox{\tiny{imp}}}(G)$, $\Vu_1\in L^2(\Omega; \widetilde{\VH}_{\sharp}(\curl, \Sigma^*))$, and $u_2\in L^2(\Omega; H^1_{\sharp, 0} (\Sigma^*))$ such that, up to a subsequence,
\begin{align*}
\chi_{G\setminus \Sigma_\delta} \Vu_\delta\twosc \chi_{\Sigma^*}(\Vu+\nabla_y u_2), \qquad \chi_{G\setminus\Sigma_\delta}\curl \Vu_\delta\twosc \chi_{\Sigma^*}(\curl \Vu +\curl_y \Vu_1).
\end{align*}

The uniform a priori bound of $\Vu_\delta$ furthermore imply that there is $\widetilde{\Vu}_0\in L^2(\Omega; H_\sharp(\curl, \Sigma))$ such that, up to a subsequence,
\[ \chi_{\Sigma_\delta} \Vu_\delta\twosc \chi_\Sigma\widetilde{\Vu}_0, \qquad \delta\chi_{\Sigma_\delta} \curl \Vu_\delta\twosc \chi_\Sigma \curl_y \widetilde{\Vu}_0,\]
cf.\ \cite{CC15hommaxwell}.
Using all these two-scale convergences, we can deduce for any $\Vpsi\in C^\infty_0(\Omega; C^\infty_\sharp(Y))$
\begin{align*}
\int_\Omega\int_\Sigma\curl_y \widetilde{\Vu}_0\cdot \Vpsi&\longleftarrow\int_\Omega\delta\curl\Vu_\delta\cdot \Vpsi\bigl(x, \frac{x}{\delta}\bigr)\\*
&\qquad\quad=\int_\Omega\delta\Vu_\delta\cdot \curl_y\Vpsi\bigl(x, \frac{x}{\delta}\bigr)\longrightarrow\int_\Omega\int_Y\curl_y\Vpsi\cdot(\chi_\Sigma\widetilde{\Vu}_0+\chi_{\Sigma^*}(\Vu+\nabla_y u_2)).
\end{align*}
Integrating now by parts on the right-hand side, we derive the continuity of the tangential traces over $\partial\Sigma$, i.e.\
\begin{align*}
\int_\Omega\int_{\partial \Sigma}\widetilde{\Vu}_0\times \Vn \cdot \Vpsi =\int_\Omega\int_{\partial \Sigma} (\Vu +\nabla_y u_2)\times \Vn\cdot \Vpsi \quad \forall \Vpsi\in C_0^\infty(\Omega; C^\infty_\sharp(Y)).
\end{align*}
Therefore, there exists $\Vu_3\in L^2(\Omega; \VH_0(\curl, \Sigma))$ such that 
\[ \Vu_\delta\twosc \Vu+\chi_{\Sigma^*}\nabla_y u_2 +\chi_\Sigma \Vu_3.\]

\emph{Third step: two-scale equation and uniqueness. }
The two-scale equation follows now from the two-scale limits by inserting a test function of the form $\Vpsi(x)+\delta \Vpsi_1(x, \frac{x}{\delta})+\nabla_y \psi(x, \frac{x}{\delta})+\Vpsi_3(x, \frac{x}{\delta})$ with smooth and periodic (in the second variable $y$) functions $\Vpsi_i$ and  with $\Vpsi_3(\cdot, y)=0$ for $y\in \Sigma^*$ and $\nabla_y \psi_2(\cdot, y)=0$ for $y\in \Sigma$ into \eqref{eq:scatteringweak}.
Uniqueness of this problem can either be derived by the uniqueness of the effective equation (see Theorem \ref{thm:effectiveeq}) or by inserting appropriate test functions.

\emph{Fourth step: $L^2(G)$ bound on $\Vu_\delta$. }
Finally, the assumption that $\Vu_\delta$ is uniformly bounded in $L^2(G)$ is proved by a contradiction argument, for details we refer to \cite{BBF15hommaxwell}.
Note that we cannot argue in the same way as for Helmholtz problems in \cite{BF04homhelmholtz, OV16hmmhelmholtz} since weak convergence in $\VH(\curl)$ does {\itshape not} imply strong convergence in $L^2$.
\end{proof}

\subsection{Stability of the Maxwell scattering problem}
\label{subsec:stabilproof}
This section is devoted to a detailed proof of Theorem \ref{thm:stabilityeff}.
First, we show the (general) stability result for real- and matrix-valued Lipschitz coefficients, Proposition \ref{prop:stabilitylipschitz}.
The discontinuity in $(\varepsilon^{-1})_{\hom}$ is then accounted for by an approximation procedure, while the partly complex $\mu_{\hom}$ can be treated more directly.

The proof uses Rellich-Morawetz identities for Maxwell's equations, see \cite{Moiola} for the constant coefficient case.
For our Lipschitz continuous coefficients, we have the following result.

\begin{lemma}
Let $G$ be an open, bounded domain, which is star-shaped w.r.t.\ a ball centered at the origin.
Let $A, B\in W^{1, \infty}(G)$  be symmetric positive definite such that $DA\cdot \Vx$ is negative semi-definite, $DB\cdot \Vx$ is positive semi-definite and that $A=\alpha\Id$ and $B=\beta\Id$ in a neighborhood of the boundary $\partial G$.
\begin{itemize}
\item If $\Vxi\in \VH(\Div, G)$ with $\curl(A\Vxi)\in L^2(G)$ and $\Vxi_T\in L^2_T(\partial G)$, then
\begin{equation}
\label{eq:curlRellich3}
\|A^{1/2}\Vxi\|^2_{L^2(G)}\leq 2\Bigl|\int_G\curl(A\Vxi)\cdot (\Vxi^*\times \Vx)+(A\Vxi\cdot \Vx)\Div\Vxi^*\Bigr|+C(G)\int_{\partial G}\alpha|\Vxi_T|^2.
\end{equation}
\item If $\Vxi\in \VH_{\mbox{\tiny{imp}}}(G)$ with $\Div(B\Vxi)\in L^2(G)$, then
\begin{equation}
\label{eq:idRellich3}
\|B^{1/2}\Vxi\|^2_{L^2(G)}\leq 2\Bigl|\int_G\curl\Vxi^*\cdot (B\Vxi\times \Vx)+(\Vxi^*\cdot \Vx)\Div(B\Vxi)\Bigr|+C(G)\int_{\partial G}\beta|\Vxi_T|^2.
\end{equation}
\end{itemize}
\end{lemma}

\begin{proof}
We only prove \eqref{eq:curlRellich3}, the procedure for \eqref{eq:idRellich3} is similar.

\emph{First step: } Assuming that $A$ and $\Vxi$ are $C^1$, we derive the point-wise identity
\begin{equation}
\label{eq:curlRellich1}
\begin{split}
2\Re\bigl(\curl(A\Vxi)\cdot (\Vxi^*\times \Vx)\bigr)&=2\Re\bigl(\Div((A\Vxi\cdot \Vx)\Vxi^*)-(A\Vxi\cdot \Vx)\,\Div\Vxi^*\bigr)\\
&\quad-\Div((A\Vxi\cdot \Vxi^*)\, \Vx)+A\Vxi\cdot \Vxi^*-(DA\cdot \Vx)\Vxi\cdot \Vxi^*,
\end{split}
\end{equation}
This is a direct computation using product rules for $\curl(\Va\times \Vb)$, $\Div(\Va\times \Vb)$, the vector calculus identity $\Va\times (\Vb\times \Vc)=(\Va\cdot \Vc)\Vb-(\Va\cdot \Vb)\Vc$ and
\[2\Re(A\Vxi\cdot (\Vx\cdot \nabla)\Vxi^*)=\Vx\cdot \nabla(A\Vxi\cdot \Vxi^*)-(DA\cdot \Vx)\Vxi\cdot \Vxi^*=\Div((A\Vxi\cdot \Vxi^*)\Vx)-3A\Vxi \cdot \Vxi^*-(DA\cdot \Vx)\Vxi\cdot \Vxi^*.\] 

\emph{Second step: } We then integrate \eqref{eq:curlRellich1} over $G$ with partial integration in the divergence-terms.
Splitting the vector $\Vxi$ in its tangential and normal components, $\Vxi_T$ and $\Vxi_N$, respectively, and using their orthogonality, we obtain
\begin{equation}
\label{eq:curlRellich2}
\begin{split}
&\!\!\!\!\int_G A\Vxi\cdot \Vxi^*-(DA\cdot \Vx)\Vxi\cdot \Vxi^*\\
&=2\Re\Bigl(\int_G\curl(A\Vxi)\cdot (\Vxi^*\times \Vx))+(A\Vxi\cdot \Vx)\, \Div\Vxi^*\Bigr)-2\Re\int_{\partial G}((A\Vxi)_T\cdot \Vx_T)\, (\Vxi^*\cdot \Vn)\\
&\quad+\Re\int_{\partial G}((A\Vxi)_T\cdot \Vv^*_T-(A\Vxi)_N\cdot \Vxi_N^*)\, (\Vx\cdot \Vn).
\end{split}
\end{equation}

\emph{Third step: } Using the assumptions of this lemma in \eqref{eq:curlRellich2} gives
\begin{align*}
\|A^{1/2}\Vxi\|^2_{L^2(G)}&\leq 2\Bigl|\int_G\curl(A\Vxi)\cdot (\Vxi^*\times \Vx)+(A\Vxi\cdot \Vx)\Div\Vxi^*\Bigr|\\
&\quad+\int_{\partial G}\alpha(|\Vxi_T|^2-|\Vxi_N|^2)(\Vx\cdot \Vn)-2\Re\int_{\partial G}\alpha(\Vxi_T\cdot \Vx_T)(\Vxi^*\cdot \Vn^*).
\end{align*}
Now we employ Young's inequality with weight $\Vx\cdot \Vn$ to the last term and obtain
\begin{align*}
\|A^{1/2}\Vxi\|^2_{L^2(G)}&\leq 2\Bigl|\int_G\curl(A\Vxi)\cdot (\Vxi^*\times \Vx)+(A\Vxi\cdot \Vx)\Div\Vxi^*\Bigr|+\int_{\partial G}\alpha|\Vxi_T|^2|\Vx|^2,
\end{align*}
which directly yields \eqref{eq:curlRellich3}.
The claim can now be obtained by approximating $A$ and $\Vxi$ with sufficiently smooth fields.
\end{proof}
For this lemma it is essential that $A$ and $B$ reduce to scalar values near the boundary because otherwise no connection between $(A\Vxi)_T$ and $\Vxi_T$ etc.\ can be drawn.
The previous lemma eliminated all terms with normal components on the boundary, which is necessary in order to apply it to functions in $\VH_{\mathrm{imp}}$.
In other words, we do not have any knowledge about $\Vv_N$ on $\partial G$ for the solution $\Vv$ to \eqref{eq:scatteringweakstabil}.

\begin{proof}[Proof of Proposition \ref{prop:stabilitylipschitz}]
We test \eqref{eq:scatteringweakstabil} with $\Vpsi=\Vv$ and take the imaginary part to obtain
\begin{equation}
\label{eq:estimatebdry}
k\|\Vv_T\|^2_{L^2(\partial G)}\leq C(\|\Vf\|_{L^2(G)}\|\Vv\|_{L^2(G)}+k^{-1}\|\Vg\|^2_{L^2(\partial G)})
\end{equation}
with a constant independent of $k$.
Next, we observe that by testing with $\nabla \phi$ for $\phi\in H^1(G)$ and constant on $\partial G$, we deduce $\Div(B\Vv)=0$.
We now apply \eqref{eq:curlRellich3} with $\xi=\curl\Vv$ and \eqref{eq:idRellich3} with $\xi=\Vv$ and obtain
\begin{align*}
&\!\!\!\!\|\Vv\|^2_{\curl, k, G}\\
&\leq 2\Bigl|\int_G\curl(A\Vv)\cdot (\curl\Vv^*\times \Vx)+k^2\curl\Vv^*\cdot (B\Vv\times \Vx)\bigr|+C\int_{\partial G}\alpha |\curl\Vv_T|^2+\beta k^2|\Vv_T|^2\\
&= 2\Bigl|\int_G\curl(A\Vv)\cdot (\curl\Vv^*\times \Vx)-k^2B\Vv\cdot (\curl\Vv^*\times \Vx)\bigr|+C\int_{\partial G}\alpha |\curl\Vv_T|^2+\beta k^2|\Vv_T|^2\\
&\leq 2\Bigl|\int_G \Vf\cdot (\curl\Vv^*\times \Vx)\Bigr|+C\int_{\partial G}k^2 |\Vv_T|^2+|\Vg|^2,
\end{align*}
where we used (the strong form of) the PDE and the boundary condition. 
Inserting H\"older's and Young's inequality for the first term on the right-hand side, we deduce
\[\|\Vv\|_{\curl, k, G}^2\leq C(\|\Vf\|^2_{L^2(G)}+\|\Vg\|^2_{L^2(\partial G)}+k^2\|\Vv_T\|^2_{L^2(\partial G)}).\]
Now plugging in \eqref{eq:estimatebdry} and using once more Young's inequality we finally obtain the asserted estimate \eqref{eq:stability}.
\end{proof}

The presented proof thus generalizes the result of \cite{Moiola} to a wider class of non-constant coefficients.

\begin{proof}[Proof of Theorem \ref{thm:stabilityeff}]
Let $\widetilde{\Vu}\in\VH_{\mathrm{imp}}$ be the solution to \eqref{eq:effectiveeq} with $\mu_{\hom}$ replaced by $\tilde{\mu}=\Id$ on all of $G$.
Using the higher regularity of $\widetilde{\Vu}$ (see Proposition \ref{prop:regscattering}), an approximation argument for $(\varepsilon^{-1})_{\hom}$, similar to \cite{OV16hmmhelmholtz}, gives the following stability
\[\|\widetilde{\Vu}\|_{\mathrm{imp}, k, G}\lesssim \|\Vf\|_{L^2(G)}+\|\Vg\|_{L^2(\partial G)}+k^{-1}\|\Vg\|_{\VH^{s_g}(\partial G)}.\]
This also implies that the inf-sup-constant behaves like $k^{-1}$, so that the above stability estimate holds also for $\tilde{\Vf}\in L^2(G)$ without the divergence-free constraint.

The difference function $\Vu-\widetilde{\Vu}$ solves \eqref{eq:effectiveeq} with $\mu_{\hom}$ replaced by $\tilde{\mu}$ and right-hand side (volume term) $k^2(\tilde{\mu}-\mu_{\hom})\Vu\in L^2(G)$.
Note that the right-hand side vanishes outside $\Omega$.
Hence, the previous arguments together with the triangle inequality yield
\[\|\Vu\|_{\mathrm{imp}, k, G}\lesssim \|\Vf\|_{L^2(G)}+\|\Vg\|_{L^2(\partial G)}+k^{-1}\|\Vg\|_{\VH^{s_g}(\partial G)}+k^2\|\Vu\|_{L^2(\Omega)}.\]

It thus remains to bound $\|\Vu\|_{L^2(\Omega)}$.
Inserting $\Vpsi=\Vu$ into \eqref{eq:effectiveeq} and considering the imaginary part gives
\[k^2c_0\|\Vu\|^2_{L^2(\Omega)}\lesssim k^{-1}\|\Vg\|^2_{L^2(\partial G)}+k^{-2}c_0^{-1}\|\Vf\|^2_{L^2(\Omega)}+\|\Vf\|_{L^2(G\setminus\overline{\Omega})}\|\Vu\|_{L^2(G\setminus\overline{\Omega})},\]
where $c_0$ denotes the lower bound on $\Im(\mu_{\hom})$.
Together with Young's inequality and the foregoing estimates this finally gives 
\[\|\Vu\|_{\mathrm{imp}, k, G}\lesssim c_0^{-1}\|\Vf\|_{L^2(\Omega)}+k c_0^{-1}\|\Vf\|_{L^2(G\setminus \overline{\Omega})}+k^{1/2}c_0^{-1/2}\|\Vg\|_{L^2(\partial G)}+k^{-1}\|\Vg\|_{\VH^{s_g}(\partial G)}.\]
Setting $c_0=k^{-2}$ according to Proposition \ref{prop:effective} finishes the proof.
\end{proof}

The proof shows that if the lower bound $c_0$ on $\Im(\mu_{\hom})$ is independent of $k$, we get the improved stability estimate
\[\|\Vu\|_{\mathrm{imp}, k, G}\lesssim \|\Vf\|_{L^2(\Omega)}+k\|\Vf\|_{L^2(G\setminus \overline{\Omega})}+k^{1/2}\|\Vg\|_{L^2(\partial G)}+k^{-1}\|\Vg\|_{\VH^{s_g}(\partial G)}.\]

\subsection{Proofs concerning the HMM}
\label{subsec:hmmproof}
In this section, we prove our central results, namely Theorems \ref{thm:infsupcond} and \ref{thm:dualestimates}.

We introduce the following dual problem: 
For $\Vf\in \VH(\Div, G)$ and $\Vf_3\in L^2(\Omega; \VH(\Div, \Sigma))$ with $\Div\Vf=0$ and $\Div_y \Vf_3=0$, find $\underline{\Vw}=(\Vw, \Vw_1, w_2, \Vw_3)\in \CH$ such that
\begin{equation}
\label{eq:dualproblem1}
\CB(\underline{\Vpsi}, \underline{\Vw})=\int_G\int_Y(\Vf+\chi_\Sigma\Vf_3)\cdot (\Vpsi+\chi_\Sigma\Vpsi_3)^*\qquad \forall \underline{\Vpsi}=(\Vpsi, \Vpsi_1, \psi_2, \Vpsi_3)\in\CH.
\end{equation}
Dual problem \eqref{eq:dualproblem1} is very similar to the two-scale limit equation \eqref{eq:twoscaleeq} and we thereby know that it is uniquely solvable.
Note that we can also apply our theory from Section \ref{sec:analysis}, in particular Assumption \ref{ass:polstable}, since the right-hand side is divergence-free.
We have the following approximation result for the dual problem.

\begin{lemma}
\label{lem:adjapprox}
Under Assumption \ref{ass:polstable}, the solution $\underline{\Vw}\in \CH$ to \eqref{eq:dualproblem1} satisfies
\begin{equation}
\label{eq:adjapprox}
\begin{split}
\inf_{\Vw_{H,h}\in \VV_{H,h}}\|\underline{\Vw}-\Vw_{H,h}\|_e&\leq C_{\appr} \bigl(k^{q+1}(H^s+h^{t_1}+h^{t_2}+ h^{t_3})\\
&\qquad\qquad\quad+k^{q+1/2}H^{s-1/2}\bigr)\|\Vf+\chi_\Sigma\Vf_3\|_{L^2(G\times Y)}.
\end{split}
\end{equation}
\end{lemma}

\begin{proof}
Interpolation estimates and best-approximation results in $\VH_{\mathrm{imp}}$, see \cite{EG15intpolbestapprox} and \cite{GM12maxwellimpedance}, yield
\begin{align*}
\inf_{\Vw_{H,h}\in \VV_{H,h}}\|\underline{\Vw}-\Vw_{H,h}\|_e&\lesssim (H^s+h^{t_1}+h^{t_2}+h^{t_3})\|\underline{\Vw}\|_{k, \CH^{s,t}}\\
&\quad+k^{1/2}H^{s-1/2}(\|\Vw_T\|_{\VH^s_\parallel(\partial G)}+\|\curl_{\partial G}\Vw_T\|_{L^2(\partial G)}),
\end{align*}
where we abbreviated by $\|\cdot\|_{k, \CH^{s,t}}$ the (weighted) higher order norms.
Inserting the regularity and stability results from Section \ref{sec:analysis} and using Assumption \ref{ass:polstable} finishes the proof.
\end{proof}

With these preliminaries, we can now prove the inf-sup-condition and the quasi-optimality of Theorem \ref{thm:infsupcond}.

\begin{proof}[Proof of Theorem \ref{thm:infsupcond}]
\emph{Proof of \eqref{eq:discretetwosceq}: }
Let $\Vv_{H,h}\in \VV_{H,h}$ be arbitrary and apply the Helmholtz decomposition to $\Vv_H=\Vz+\nabla \theta$ and $\Vv_{h,3}=\Vz_3+\nabla_y \theta_3$.
We write in short $\Vv_{H,h}=\underline{\Vz}+\nabla \Vtheta$ with $\underline{\Vz}=(\Vz, \Vv_{h,1}, 0, \Vz_3)$ and $\nabla \Vtheta:=(\nabla \theta, 0, \nabla_y v_{h,2}, \nabla_y \theta_3)$.
Let $\underline{\Vw}=(\Vw, \Vw_1, w_2, \Vz_3)\in \CH$ be the solution to dual problem \eqref{eq:dualproblem1} with right-hand side $C_gk^2(\Vz+\chi_\Sigma\Vz_3)$.
Let $\Vw_{H,h}$ be the best-approximation to $\underline{\Vw}$ in the two-scale energy norm $\|\cdot\|_e$.

Imitating the proof of the analytical inf-sup condition in Lemma \ref{lem:infsupconst}, we would like to choose the test function $F(\Vv_{H,h})+\Vw_{H,h}$.
Unfortunately, $F(\Vv_{H,h})$ is not discrete any more, so that we have to apply an additional interpolation operator. 
We choose the corresponding standard (nodal) interpolation operator for each of the single spaces of $\VV_{H,h}$ and call the resulting operator $I_{H,h}$.
Hence, we obtain
\begin{align*}
&\!\!\!\!\bigl|\CB(\Vv_{H,h}, I_{H,h}(F(\Vv_{H,h}))+\Vw_{H,h})\bigr|\\
&\geq\bigl|\CB(\Vv_{H,h}, F(\Vv_{H,h})+\underline{\Vw})\bigr|-\bigl|\CB(\Vv_{H,h}, (I_{H,h}-\mathrm{id})F(\Vv_{H,h})\bigr|-\bigl|\CB(\Vv_{H,h}, \Vw_{H,h}-\underline{\Vw})\bigr|.
\end{align*}
The first term can be estimated as
\begin{align*}
\bigl|\CB(\Vv_{H,h}, F(\Vv_{H,h})+\underline{\Vw})\bigr|&=\bigl|\CB(\Vv_{H,h}, F(\Vv_{H,h})+C_gk^2(\Vz+\chi_{\Sigma}\Vz_3, \Vv_H+\chi_{\Sigma}\Vv_{h,3})\bigr|\\
&=\bigl|\CB(\Vv_{H,h}, F(\Vv_{H,h})+C_gk^2\|\Vz+\chi_{\Sigma}\Vz_3\|^2_{L^2(G\times Y)}\bigr|\geq \gamma_{\mathrm{ell}}\|\Vv_{H,h}\|^2_e.
\end{align*}

Using the continuity of $\CB$ and Lemma \ref{lem:adjapprox}, we deduce for the third term
\begin{align*}
&\!\!\!\!\bigl|\CB(\Vv_{H,h}, \Vw_{H,h}-\underline{\Vw})\bigr|\\*
&\leq C_cC_{\appr}C_g(k^{q+2}(H^s+h^{t_1}+h^{t_2}+h^{t_3})+k^{q+3/2}H^{s-1/2})\|\Vv_{H,h}\|_e\, k\|\Vz+\chi_\Sigma\Vz_3\|_{L^2(G\times Y)}\\
&\leq C_cC_{\appr}C_g(k^{q+2}(H^s+h^{t_1}+h^{t_2}+h^{t_3})+k^{q+3/2}H^{s-1/2})\|\Vv_{H,h}\|^2_e,
\end{align*}
where we used the stability of the Helmholtz decomposition in the last step.

For the second term we note that $F(\Vv_{H,h})=2\underline{\Vz}-\Vv_{H,h}$.
It holds that $\curl(I_{H,h}-\mathrm{id})\underline{\Vz}=0$ because the nodal interpolation operator is a commuting projector and $\curl\underline{\Vz}=\curl\Vv_{H,h}$.
In particular, this means that the curl and the tangential trace of $\Vz+\chi_\Sigma\Vz_3$ are discrete functions, so that we can apply the modified interpolation estimates \cite[Lemmas 5.1 and 5.3]{GM12maxwellimpedance}.
This yields for the second term
\begin{align*}
&\!\!\!\!\bigl|\CB(\Vv_{H,h}, (I_{H,h}-\mathrm{id})F(\Vv_{H,h}))\bigr|\\
&\leq 2C_c\|\Vv_{H,h}\|_e\, \bigl(k\|(I_{H,h}-\mathrm{id})(\Vz+\chi_\Sigma\Vz_3)\|_{L^2(G\times Y)}+k^{1/2}\|(I_{H,h}-\mathrm{id})\Vz_T\|_{L^2(\partial G)}\bigr)\\
&\leq 2C_cC_{\appr}(k (H^s+h^{t_3})+k^{1/2}H^{s-1/2})\|\Vv_{H,h}\|_e,
\end{align*}
where we used the higher regularities of the decomposition from Remark \ref{rem:reghelmholtzdecomp}.
The second term thus is of lower order than the third term and can be absorbed in the latter because of $k\geq k_0$.
All in all, this gives
\begin{align*}
&\!\!\!\!\bigl|\CB(\Vv_{H,h}, I_{H,h}(F(\Vv_{H,h}))+\Vw_{H,h})\bigr|\\
&\geq (\gamma_{\mathrm{ell}}-C_cC_{\appr}(C_g+2)(k^{q+2}(H^s+h^{t_1}+h^{t_2}+h^{t_3})+ k^{q+3/2}H^{s-1/2}))\|\Vv_{H,h}\|^2_e\\
&\geq \gamma_{\mathrm{ell}}/2\|\Vv_{H,h}\|^2_e,
\end{align*}
where we used the resolution condition \eqref{eq:resolcond1} in the last step.

Furthermore, it holds -- with the same arguments as before -- that
\begin{align*}
&\!\!\!\!\|I_{H,h}(F(\Vv_{H,h}))+\Vw_{H,h}\|_e\\
&\leq \|F(\Vv_{H,h})\|_e+\|\underline{\Vw}\|_e+\|\underline{\Vw}-\Vw_{H,h}\|_e+\|(I_{H,h}-\mathrm{id})F(\Vv_{H,h})\|_e\\
&\leq (1+C_g C_{\stab, e}k^{q+1}+(C_g+2) C_{\appr}(k^{q+2}(H^s+h^{t_1}+h^{t_2}+h^{t_3})+k^{q+3/2}H^{s-1/2}))\|\Vv_{H,h}\|_e^2,
\end{align*}
which finishes the proof of the inf-sup-condition.

\emph{Proof of the quasi-optimality \eqref{eq:quasiopt}:}
Let $\Ve:=(e_0, e_1, e_2, e_3)$ and apply the Helmholtz decomposition \eqref{eq:decomp} to $e_0+\chi_\Sigma e_3=\Vz+\chi_\Sigma \Vz_3+\nabla \theta+\chi_\Sigma\nabla_y \theta_3$.
We write in short $\Ve=\underline{\Vz}+\nabla \Vtheta$ with $\underline{\Vz}=(\Vz, e_1, 0, \Vz_3)$ and $\nabla \Vtheta:=(\nabla \theta, 0, \nabla_y e_2, \nabla_y \theta_3)$.

Using the G{\aa}rding-type inequality \eqref{eq:gardingtwosc}, we have that
\begin{equation}
\label{eq:gaardingerror}
\begin{split}
\gamma_{\mathrm{ell}}\|\Ve\|^2_e&\leq |\CB(\Ve, F(\Ve)) +C_gk^2\|\Vz+\chi_\Sigma\Vz_3\|^2_{L^2(G\times Y)}|\leq |\CB(\Ve, \Ve)|+(C_g+2)k^2\|\Vz+\chi_{\Sigma}\Vz_3\|^2_{L^2(G\times Y)}.
\end{split}
\end{equation}

The main work is now to bound the second term.
For this, let $\underline{\Vw}\in \CH$ be the solution to dual problem \eqref{eq:dualproblem1} with right-hand side $\Vz+\chi_\Sigma\Vz_3$.
Because of the orthogonality in the Helmholtz decomposition of \eqref{eq:decomp} it holds that
\[k\|\Vz+\chi_\Sigma\Vz_3\|^2_{L^2(G\times Y)}=k(\Vz+\chi_\Sigma\Vz_3, e_0+\chi_\Sigma e_3)_{L^2(G\times Y)}=k\CB(\Ve, \underline{\Vw}).\]
Using Galerkin orthogonality and Lemma \ref{lem:adjapprox}, we obtain for any $\Vw_{H,h}\in \VV_{H,h}$ that
\begin{align*}
k\|(\Vz+\chi_\Sigma \Vz_3)\|^2_{L^2}&=k\CB(\Ve, \underline{\Vw})=k\CB(\Ve, \underline{\Vw}-\Vw_{H,h})\\*
&\leq C_cC_{\appr}(k^{q+2}(H^s+h^{t_1}+h^{t_2}+h^{t_3})+k^{q+3/2}H^{s-1/2})\|(\Vz +\chi_\Sigma \Vz_3)\|_{L^2}\|\Ve\|_e
\end{align*}
and thus
\begin{equation}
\label{eq:l2estimatez}
k\|(\Vz+\chi_\Sigma \Vz_3)\|_{L^2(G\times Y)}\leq C_c C_{\appr}(k^{q+2}(H^s+h^{t_1}+h^{t_2}+h^{t_3})+k^{q+3/2}H^{s-1/2})\|\Ve\|_e.
\end{equation}

Inserting \eqref{eq:l2estimatez} now into \eqref{eq:gaardingerror} and applying Galerkin orthogonality, we get
\begin{align*}
\gamma_{\mbox{\tiny{ell}}}\|\Ve\|_e^2&\leq \bigl|\CB(\Ve, \Ve)\bigr|+(C_g+2) k^2 \|\Vz+\chi_\Sigma\Vz_3\|^2_{L^2(G\times Y)}\\
&\leq\bigl|\CB(\Ve, \underline{\Vu}-\Vv_{H,h})\bigr|+(C_g+2)\|\Ve\|_e\, k\|\Vz+\chi_\Sigma\Vz_3\|_{L^2(G\times Y)}\\
&\leq C_c\|\Ve\|_e\|\underline{\Vu}-\Vv_{H,h}\|\\
&\qquad+(C_g+2) C_cC_{\appr}(k^{q+2}(H^s+h^{t_1}+h^{t_2}+ h^{t_3})+k^{q+3/2}H^{s-1/2})\|\Ve\|_e^2,
\end{align*}
which gives the claim using resolution condition \eqref{eq:resolcond1}.
\end{proof}

The proof of the quasi-optimality already showed that the compact perturbation is of higher order (with respect to the rates in the mesh size) than the energy error. 
This kind of Aubin-Nitsche trick can be extended to the whole Helmholtz decomposition.
\begin{proof}[Proof of Theorem \ref{thm:dualestimates}]
The estimate for $\Vz+\chi_\Sigma\Vz_3$ is already given by \eqref{eq:l2estimatez} (after dividing by $k$).
To estimate $\theta+\chi_\Sigma\theta_3$, we pose another dual problem (cf.\ \cite{HOV15maxwellHMM}):
Find $\Vw:=(w, w_2, w_3)\in \CS:=H^1_{\partial G}\times L^2(\Omega; H^1_{\sharp, 0}(\Sigma^*))\times L^2(\Omega; H^1_0(\Sigma))$ such that
\begin{align*}
\CA(\Vpsi, \Vw)&:=-k^2\int_G\int_y(\nabla \psi+\chi_{\Sigma^*}\nabla_y \psi_2+\chi_\Sigma\nabla_y \psi_3)\cdot (\nabla w+\chi_{\Sigma^*}\nabla_y w_2+\chi_\Sigma\nabla_y w_3)^*\\
&=\int_G\int_Y (\theta+\chi_\Sigma\theta_3)\cdot (\psi+\chi_\Sigma\psi_3)^*\qquad \forall \Vpsi=(\psi, \psi_2, \psi_3)\in \CS.
\end{align*}
Let us denote by $\Vw_{H,h}=(w_H, w_{h, 2}, w_{h, 3})$ the solution of the corresponding discrete problem over the Lagrange finite element spaces $W_H\subset H^1_{\partial G}$, $W_h(\Sigma^*)\subset H^1_{\sharp, 0}(\Sigma^*)$, and $W_h(\Sigma)\subset H^1_0(\Sigma)$.
It is a well-known fact of finite element exterior calculus that $\nabla W_H\subset \VH_{\mathrm{imp}}$, etc.
We obtain with the Galerkin orthogonality
\begin{align*}
\|\theta+\chi_\Sigma\theta_3\|^2_{L^2(G\times Y)}&=\CA((\theta, e_2, \theta_3), \Vw)=\CB((\nabla\theta, e_1, e_2, \nabla_y \theta_3), (\nabla w, 0, w_2, \nabla_y w_3))\\
&=\CB(\Ve, (\nabla w, 0, w_2, \nabla_y w_3))-\CB((\Vz, 0, 0, \Vz_3), (\nabla w, 0, w_2, \nabla_y w_3))\\
&=\CB(\Ve, (\nabla(w-w_H), 0, w_2-w_{h,2}, \nabla_y(w_3-w_{h,3}))\\*
&\quad-\CB((\Vz, 0, 0, \Vz_3), (\nabla w, 0, w_2, \nabla_y w_3)).
\end{align*}
Using the approximation properties of the Lagrange finite element spaces and the regularity and stability of elliptic diffusion two-scale problems, we deduce
\begin{align*}
\|\theta+\chi_\Sigma\theta_3\|^2_{L^2(G\times Y)}&\lesssim \|\Ve\|_e\, k \|\nabla(w-w_H)+\chi_{\Sigma^*}\nabla_y (w_2-w_{h,2})+\chi_\Sigma\nabla_y(w_3-w_{h, 3})\|_{L^2(G\times Y)}\\*
&\quad +k^2\|\Vz+\chi_\Sigma\Vz_3\|_{L^2(G\times Y)} \|\nabla w+\chi_{\Sigma^*}\nabla_y w_2+\chi_\Sigma\nabla_yw_3\|_{L^2(G\times Y)}\\
&\lesssim (H^s+h^{t_2}+h^{t_3})\|\Ve\|_e\|\theta+\chi_\Sigma\theta_3\|_{L^2(G\times Y)}\\*
&\quad +(k^{q+2}(H^s+h^{t_1}+h^{t_2}+h^{t_3})+k^{q+3/2}H^{s-1/2})\|\Ve\|_e\|\theta+\chi_\Sigma\theta_3\|_{L^2(G\times Y)},
\end{align*}
which in combination with \eqref{eq:l2estimatez} finishes the proof.
\end{proof}

\section{Numerical results}
\label{sec:experiment}
In this section we give some numerical results on the HMM with particular respect to the convergence order (see Theorem \ref{thm:infsupcond}, Corollary \ref{cor:errorestimates} and Theorem \ref{thm:dualestimates}) and the behavior for different frequencies $k$ and different values $\mu_{\hom}$.
The implementation was done with the module {\sffamily dune-gdt} \cite{wwwdunegdt} of the DUNE software framework \cite{BB+08dune1, BB+08dune2}.

We consider the macroscopic domain $G=(0,1)^3$ with embedded scatterer $\Omega=(0.25, 0.75)^3$.
The boundary condition $\Vg$ is computed as $\Vg=\curl\Vu_{\mathrm{inc}}\times \Vn-ik\Vn\times (\Vu_{\mathrm{inc}}\times \Vn)$ with the (left-going), $\Ve_2$-polarized incoming plane wave $\Vu_{\mathrm{inc}}=\exp(-ik x_1)\Ve_2$.
The unit cube $Y$ has the inclusion $\Sigma =(0.25, 0.75)^3$ and we choose  the inverse permittivities as $\varepsilon_0^{-1}=1.0$ and $\varepsilon_1^{-1}=1.0-0.01i$.
Obviously, the real parts of both parameters are of the same order and $\varepsilon_1$ is only slightly dissipative.

First, we analyze the dependency of the effective permeability $\mu_{\hom}$ on the wavenumber $k$.
The contribution to $\mu_{\hom}$ from the second cell problem \eqref{eq:cellproblem2} in $\Sigma^*$ is independent of $k$, as expected.
The wavenumber-dependency is wholly caused by cell problem \eqref{eq:cellproblem3} inside $\Sigma$.
As discussed also in \cite{BBF15hommaxwell} and for the two-dimensional case in \cite{BF04homhelmholtz, OV16hmmhelmholtz}, significant changes in $\mu_{\hom}$ are expected around the eigenvalues of the vector Laplacian.
Only some of the eigenvalues, namely those where the mean value of the eigenfunction(s) is not the zero vector, will eventually lead to resonances in the behavior of the effective permeability.
As $\Sigma$ is a cube, those eigenvalues are explicitly known and for our setup, the first interesting values are $k\approx 8.9$ and $k\approx 19.9$. 
We compute $\mu_{\hom}$ using cell problems \eqref{eq:cellproblem2} and \eqref{eq:cellproblem3} with a mesh consisting of $196,608$ elements on $Y$.
Figure \ref{fig:muhom} depicts the behavior of the diagonal entries of $\Re(\mu_{\hom})$ and $\Im (\mu_{\hom})$ (all three diagonal entries are the same due to symmetry) for changing $k$.
As predicted, we see a significant change of behavior around the eigenvalues, where the imaginary part has large values and the real part shows resonances.
For the first eigenvalue, this resonance is strong enough to produce a negative real part, while this is not the case for the second eigenvalue in our setup.

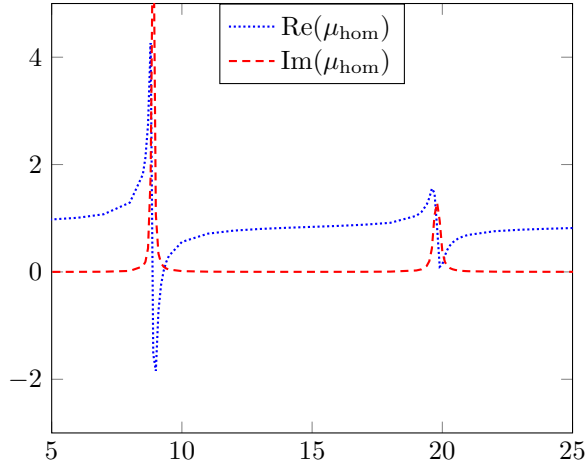
\begin{figure}
\begin{center}
\begin{tikzpicture}
\begin{axis}[axis equal image=false, legend entries={$\Re(\mu_{\hom})$, $\Im(\mu_{\hom})$}, xmin=5, xmax=25, ymin=-3, ymax=5, legend style={at={(0.5, 1.0)}, anchor=north}]
\addplot+[no markers, densely dotted, thick] table[x=omega, y= Re(mu), col sep=comma]{images/muhom.csv};
\addplot+[no markers, densely dashed, thick] table[x=omega, y= Im(mu), col sep=comma]{images/muhom.csv};
\end{axis}
\end{tikzpicture}
\end{center}
\caption{Dependence of the effective permeability $\mu_{\hom}$ on the wavenumber $k$ for square inclusion $\Sigma=(0.25, 0.75)^2$ with $\varepsilon_1^{-1}=1.0-0.01i$.}
\label{fig:muhom}
\end{figure}

We now take a closer look at the convergence of errors and verify the predictions of Theorem \ref{thm:infsupcond}/Corollary \ref{cor:errorestimates} and Theorem \ref{thm:dualestimates}.
We use a reference homogenized solution by computing the effective parameters with $196,608$ elements on $Y$ and the solving the effective homogenized equation \eqref{eq:effectiveeq} with these parameters using a mesh with $663,552$ elements for $G$.
This reference homogenized solution is compared to the macroscopic part $\Vu_H$ of the HMM-approximation on a sequence of simultaneously refined macro- and microscale meshes for the frequencies $k=9$ and $k=12$.
Note that $k=12$ corresponds to ``standard'' effective parameters, while for $k=9$, $\Re(\mu_{\hom})$ is negative definite.
The errors in the $L^2$ and $\VH(\curl)$-semi norm are shown in Table \ref{tab:homerror12} for $k=12$ and in Table \ref{tab:homerror9} for $k=9$.
In order to verify Theorem \ref{thm:dualestimates}, we compute an approximation of the gradient part $\theta$ of the Helmholtz decomposition:
We solve the Poisson problem determining $\theta$ (with right-hand side $e_0$) using linear Lagrange elements on the reference mesh (with $663,552$ elements).
The $L^2$ norms of this resulting $\theta$ are also shown in Tables \ref{tab:homerror12} and Table \ref{tab:homerror9}, respectively.
The experimental order of convergence (EOC), which is defined for two mesh sizes $H_1>H_2$ and the corresponding error values $e_{H_1}$ and $e_{H_2}$ as $EOC(e):=\ln(\frac{e_{H_1}}{e_{H_2}})/\ln(\frac{H_1}{H_2})$, verifies the linear convergence in $L^2 $ and $\VH(\curl)$, predicted in Theorem \ref{thm:infsupcond} and Corollary \ref{cor:errorestimates}, and the quadratic convergence of the Helmholtz decomposition, predicted in Theorem \ref{thm:dualestimates}.
Note that from the geometry one might expect a reduced regularity of the analytical solution and therefore, a sub-linear convergence of the $\VH(\curl)$-error.
We believe that the linear convergence observed in the experiment does not imply a sub-optimality of the error bound in Theorem \ref{thm:infsupcond}, but that in fact, the analytical homogenized solution in this special case has full $\VH^1_{pw}(\curl, G)$ regularity, probably because of the specific boundary condition.
This clearly shows that our general theory holds for all regimes of wavenumbers even if they result in unusual effective parameters.
This is consistent with the observations made for the two-dimensional case in \cite{OV16hmmhelmholtz}.

\begin{table}
\caption{Convergence history and EOC for the error between the macroscopic part $\Vu_H$ of the HMM approximation and the reference homogenized solution for $k=12$.}
\label{tab:homerror12}
\centering
\begin{tabular}{@{}ccccccc@{}}
\toprule
$H=h$&$\|e_0\|_{L^2(G)}$&$\|\curl e_0\|_{L^2(G)}$&$\|\theta\|_{L^2(G)}$&EOC($e_0$)&EOC($\curl e_0$)&EOC($\theta$)\\
\midrule
$\sqrt{3}\times 1/4$ & $0.945214$ & $11.6003$ & $0.01555$ & ---& --- & ---\\
$\sqrt{3}\times 1/8$ & $0.5316$ & $5.76452$ & $0.0096331$ &$0.8303$ & $1.0089$ & $0.6908$\\
$\sqrt{3}\times 1/12$ & $0.3211809$ & $3.36067$ & $0.00409982$ & $1.2379$ & $1.3308$ & $2.1069$\\
$\sqrt{3}\times 1/16$ & $0.230797$ & $2.38167$ & $0.00220056$ & $1.1555$ & $1.1969$ & $2.1629$\\
\bottomrule
\end{tabular}
\end{table}

\begin{table}
\caption{Convergence history and EOC for the error between the macroscopic part $\Vu_H$ of the HMM approximation and the reference homogenized solution for $k=9$.}
\label{tab:homerror9}
\centering
\begin{tabular}{@{}ccccccc@{}}
\toprule
$H=h$&$\|e_0\|_{L^2(G)}$&$\|\curl e_0\|_{L^2(G)}$&$\|\theta\|_{L^2(G)}$&EOC($e_0$)&EOC($\curl e_0$)&EOC($\theta$)\\
\midrule
$\sqrt{3}\times 1/4$ & $0.697211$ & $5.54104$ & $0.0242162$ & ---& --- & ---\\
$\sqrt{3}\times 1/8$ & $0.410991$ & $2.94379$ & $0.0104552$ &$0.7625$ & $0.9125$ & $1.2118$\\
$\sqrt{3}\times 1/12$ & $0.285927$ & $1.85786$ & $0.00574651$ & $0.8949$ & $1.1351$ & $1.4761$\\
$\sqrt{3}\times 1/16$ & $0.216505$ & $1.31478$ & $0.0033278$ & $0.9668$ & $1.2019$ & $1.8989$\\
\bottomrule
\end{tabular}
\end{table}

Finally, we compare the two frequencies $k=9$ and $k=12$ in more detail.
They have a different physical meaning: For $k=12$, normal transmission through the scatterer is expected, while $k=9$ corresponds to a wavenumber in the band gap due to the negative definite real part of $\mu_{\hom}$.
Thus, wave propagation through the scatterer is forbidden for $k=9$.
We consider the magnitude of the real part of $\Vu_H$ (the macroscopic part of the HMM-approximation with $H=h=\sqrt{3}\times 1/16$) and plot it in Figure \ref{fig:3dplots}.
The isosurfaces are almost parallel planes for $k=12$ indicating normal, almost undisturbed propagation of the wave through the scatterer.
Note that the effective wave speed inside the scatterer does not differ greatly from the one outside in our choice of material parameters.
In contrast, the scatterer has a significant influence on the wave propagation for $k=9$, as we can deduce from the distorted wavefronts in Figure \ref{fig:3dplots}, right.

\begin{figure}
\includegraphics[width=0.48\textwidth, trim= 65mm 5mm 30mm 10mm, clip=true, keepaspectratio=false]{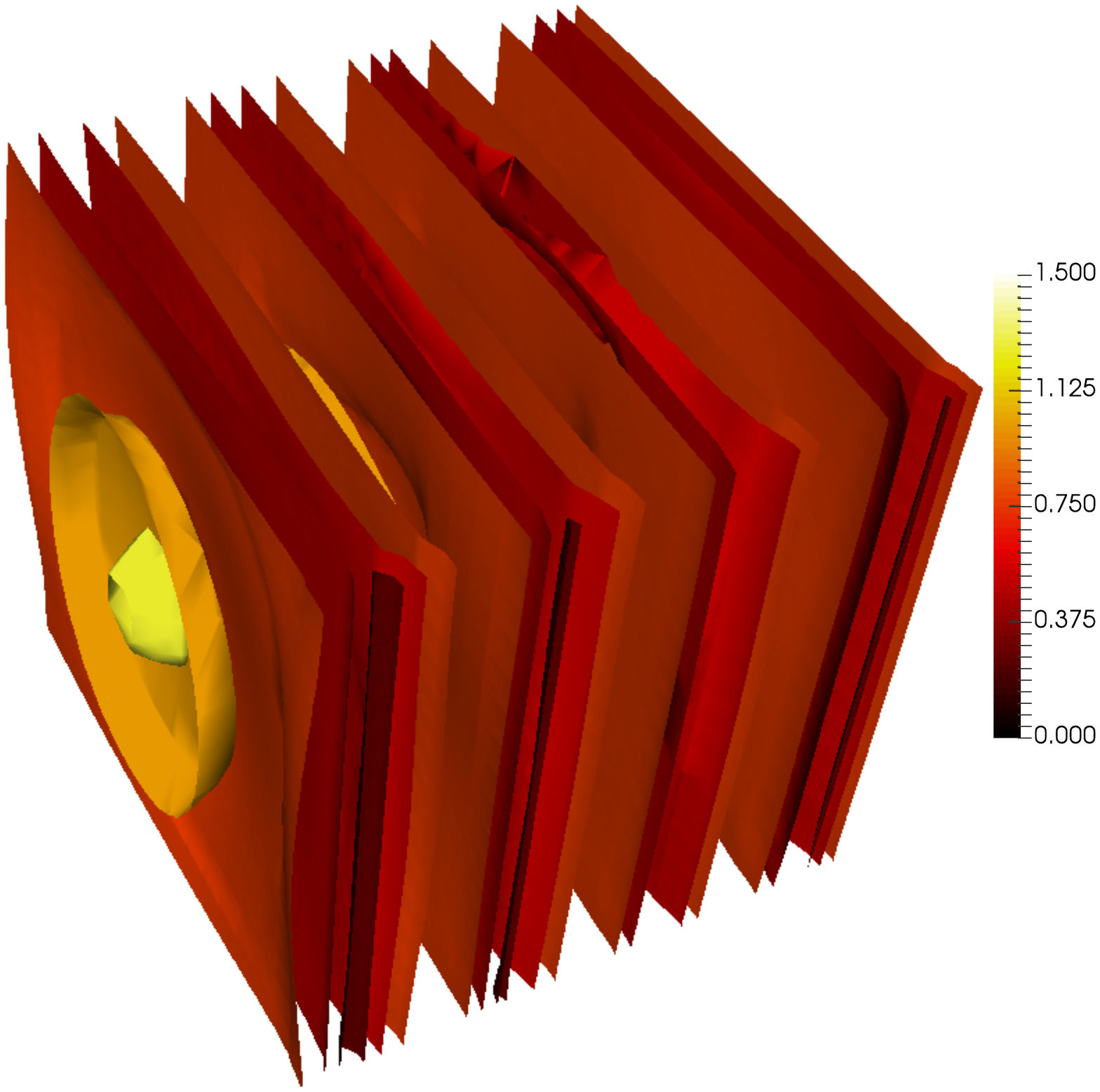}%
\hspace{0.5cm}%
\includegraphics[width=0.48\textwidth, trim= 65mm 5mm 30mm 10mm, clip=true, keepaspectratio=false]{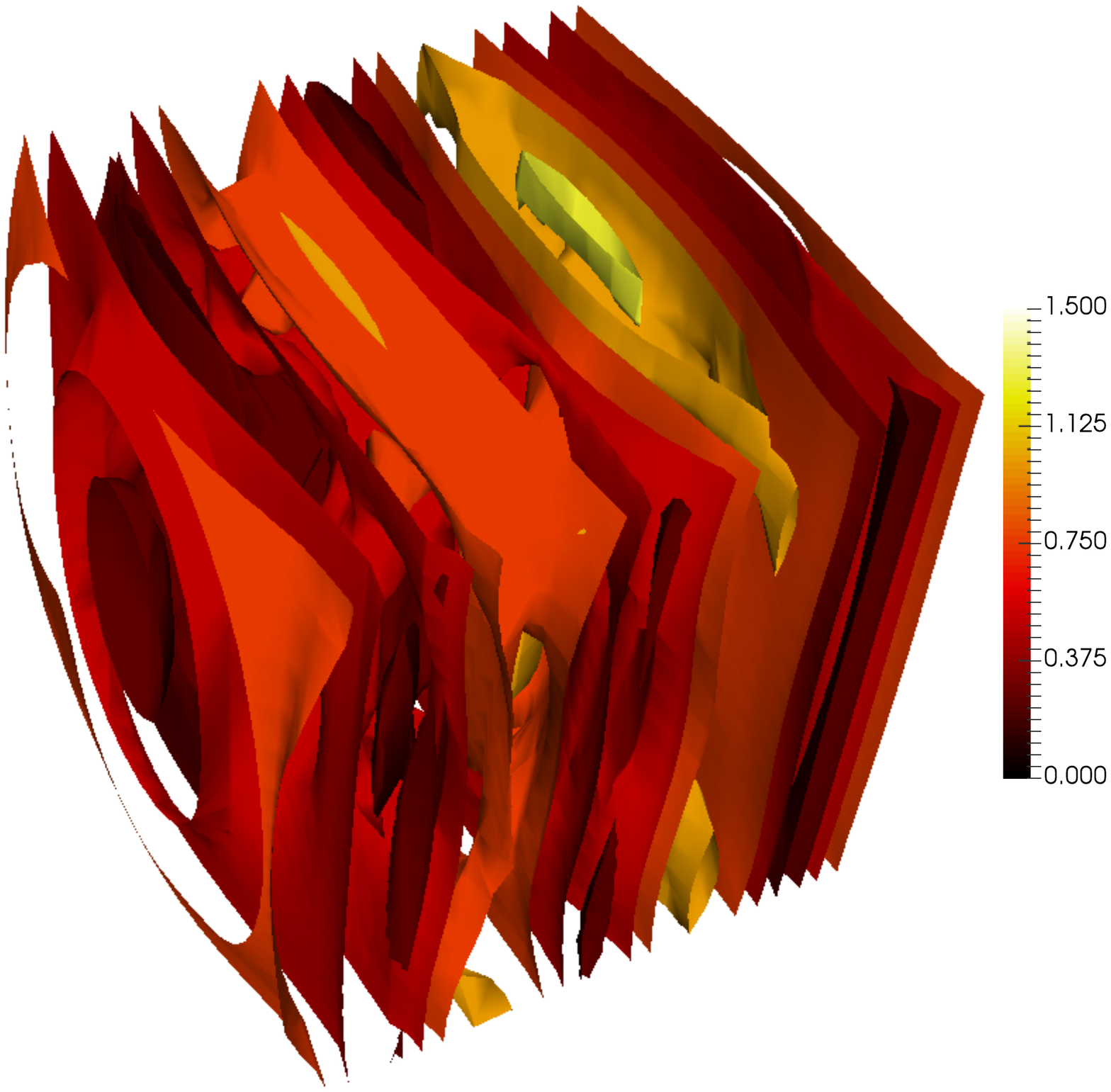}
\caption{Isosurfaces for the magnitude of $\Re(\Vu_H)$ for $k=12$ (left) and $k=9$ (right)}
\label{fig:3dplots}
\end{figure}

To compare this in more detail, we study two-dimensional representations in the plane $y=0.545$ in Figure \ref{fig:2dplots}.
There we depict the $x_2$-component, which is the principal one due to the polarization of the incoming wave.
The top row shows again the macroscopic part $\Vu_H$ of the HMM-approximation and we see the expected exponential decay of the amplitude inside the scatterer for $k=9$ (top right), while the amplitude is not affected for $k=12$.
The zeroth order approximation $\Vu^0_{\mathrm{HMM}}:=\Vu_H+\nabla_yu_{h,2}(\cdot, \frac{\cdot}{\delta})+\Vu_{h,3}(\cdot, \frac{\cdot}{\delta})$ in the bottom row of Figure \ref{fig:2dplots} explains this effect.
The (resonant) amplitudes inside the inclusions are much higher for $k=9$ than for $k=12$.
Wavenumber $k=9$ almost coincides with the eigen resonance of the inclusions, which explains the high amplitudes. This implies that a lot of the waves' energy is confined to the inclusions and thus the wave amplitude is decaying throughout the scatterer.
In contrast for the wavenumber $k=12$ the higher amplitudes inside the inclusions are solely due to the different material parameters and do not trigger any resonances, so that the overall wave propagation remains undisturbed.

\begin{figure}
\includegraphics[width=0.48\textwidth, trim= 75mm 33mm 51mm 32mm, clip=true, keepaspectratio=false]{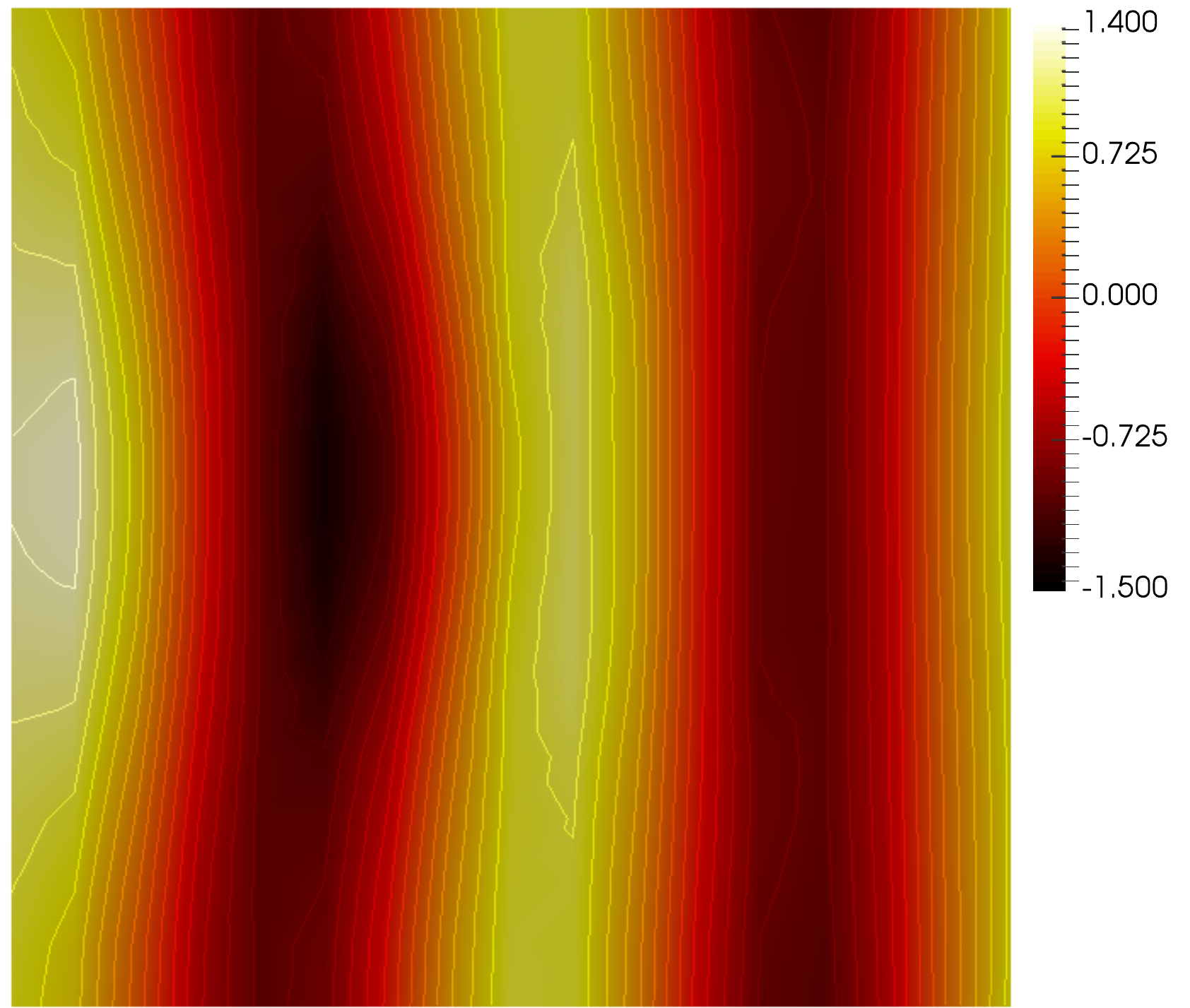}%
\hspace{0.5cm}%
\includegraphics[width=0.48\textwidth, trim= 75mm 33mm 51mm 32mm, clip=true, keepaspectratio=false]{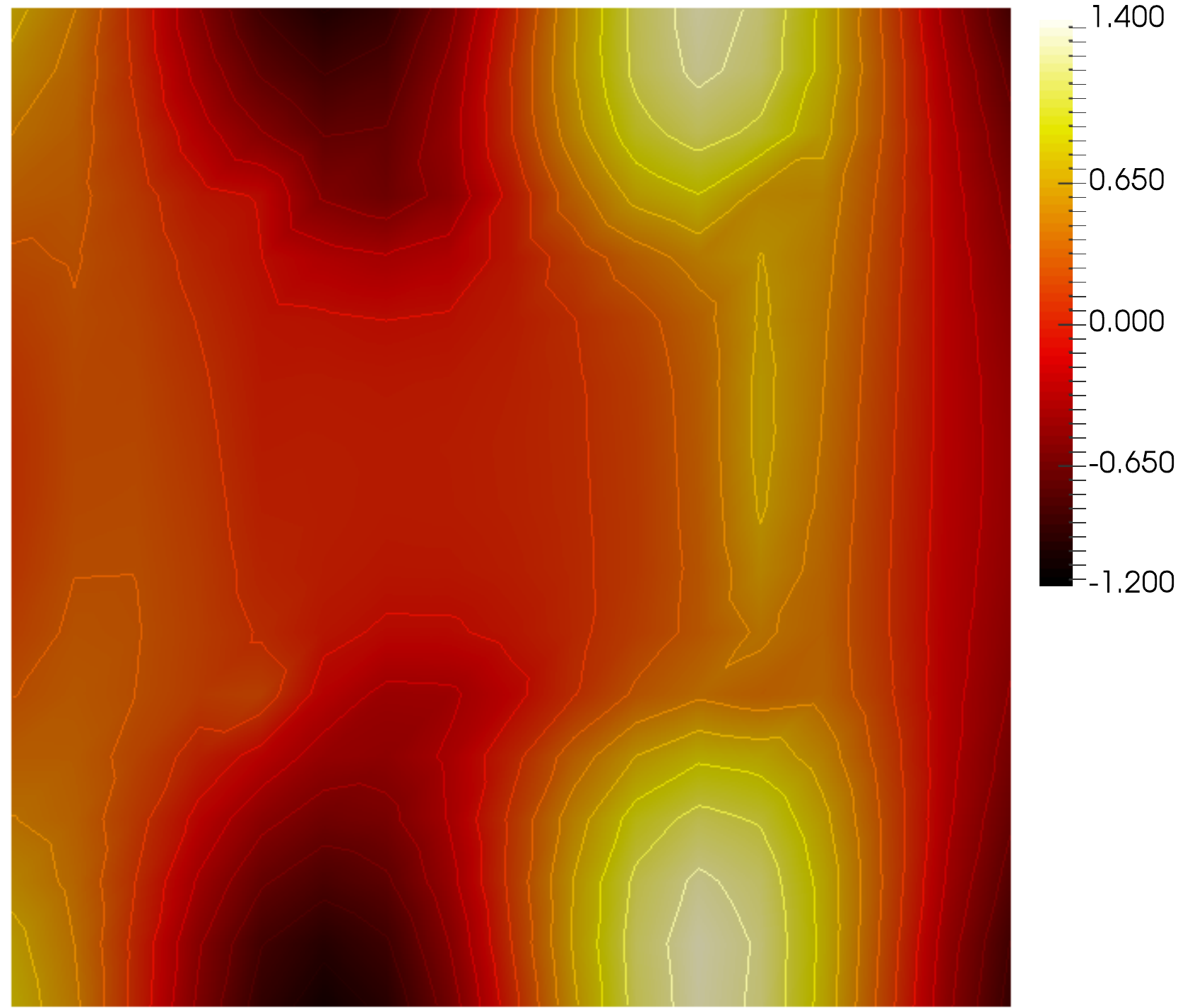}\\
\vspace{1pt}\\
\includegraphics[width=0.48\textwidth, trim= 75mm 33mm 51mm 32mm, clip=true, keepaspectratio=false]{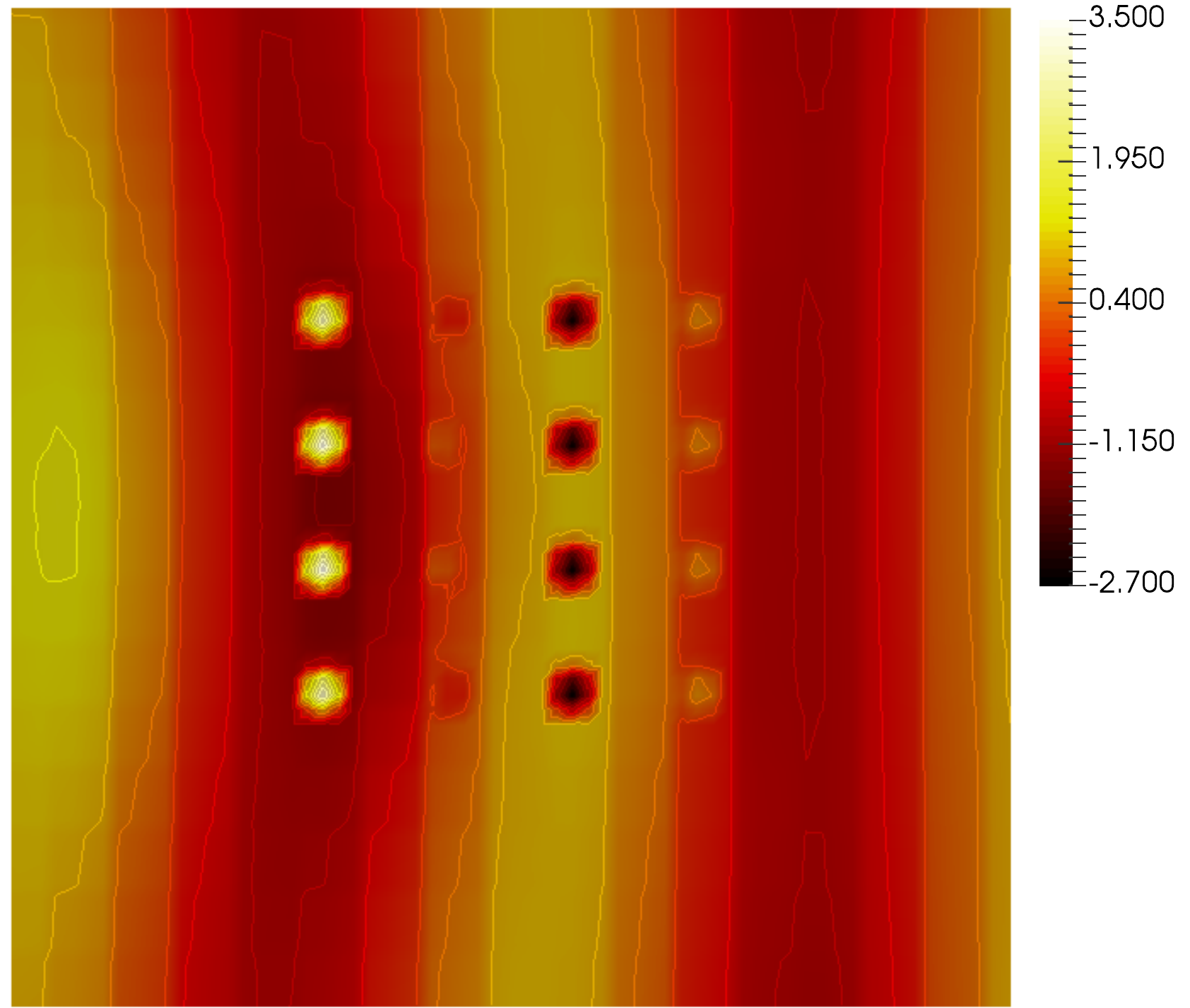}%
\hspace{0.5cm}%
\includegraphics[width=0.48\textwidth, trim= 75mm 33mm 51mm 32mm, clip=true, keepaspectratio=false]{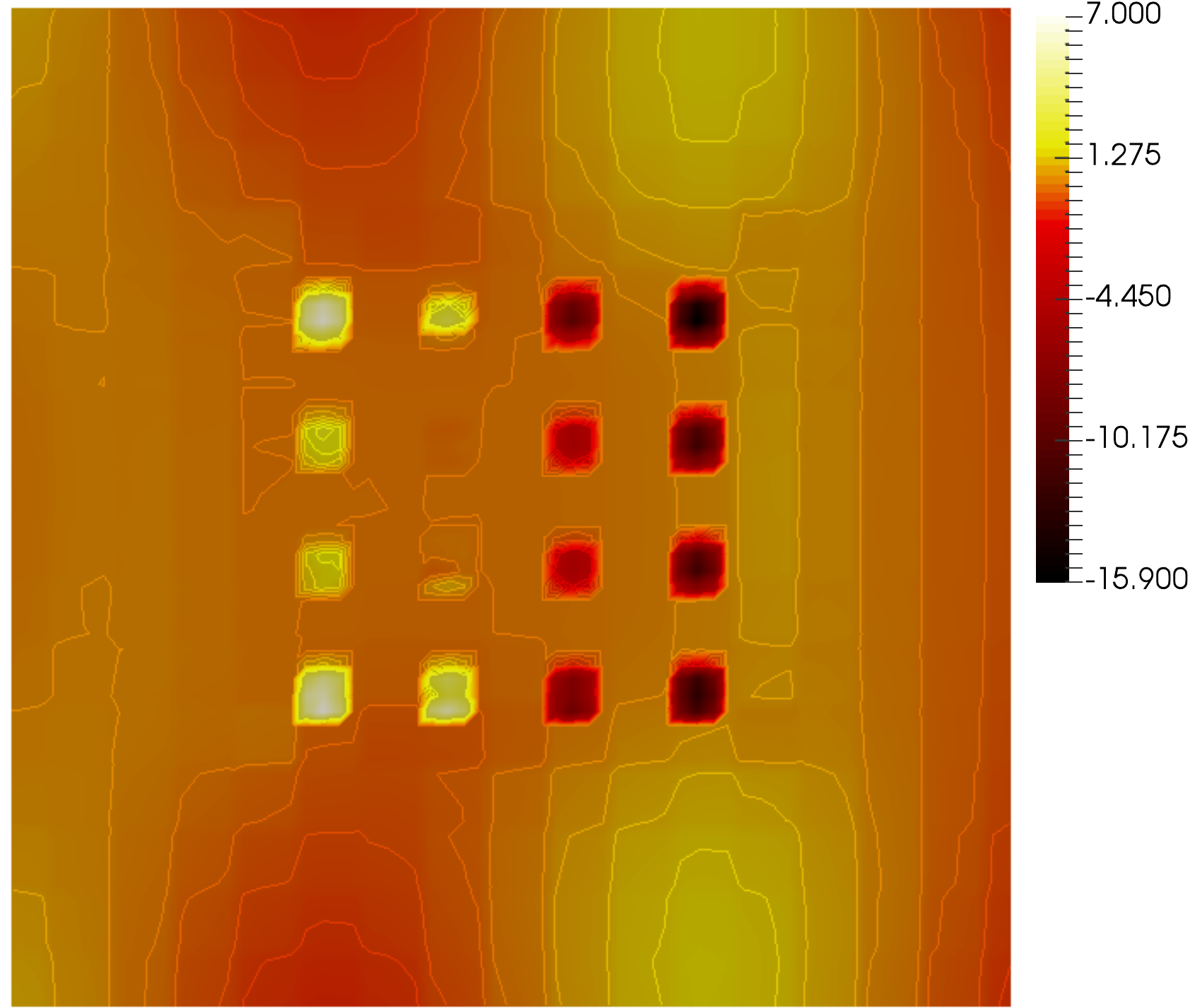}
\caption{In the plane $y=0.545$: $x_2$-component of $\Re(\Vu_H)$ (top row) and of $\Re(\Vu^0_{\mathrm{HMM}})$ (bottom row) for $k=12$ (left column) and $k=9$ (right column).}
\label{fig:2dplots}
\end{figure}

\section*{Conclusion}
We suggested a new Heterogeneous Multiscale Method (HMM) for the Maxwell scattering problem with high contrast. 
A two-scale limit problem is obtained via two-scale convergence, which is equivalent to existing homogenization results in the literature, but has some advantages for analysis and numerics.
The stability and regularity  of the homogenized system is analyzed rigorously and thereby, the first stability result for time-harmonic Maxwell's equations with impedance boundary condition and non-constant coefficients is proved.
The HMM is defined as direct finite element discretization of the two-scale equation, which is crucial for the numerical analysis. 
Well-posedness, quasi-optimality and a priori error estimates in energy and dual norms are shown under an (unavoidable) resolution condition linking the mesh size and the wavenumber and which depends on the polynomial stability. 
Numerical experiments verify the developed convergence results.
The comparison of the HMM-approximation (with the discrete correctors) to a full reference solution of the heterogeneous problem is subject of future research. 

\section*{Acknowledgment}
The author is thankful to Mario Ohlberger for fruitful discussions on the subject and his helpful remarks regarding the manuscript.

\end{document}